\documentclass[hidelinks,onefignum,onetabnum]{siamart251216}

\usepackage{amssymb,mathtools,mathrsfs,booktabs,color}
\usepackage{multirow,multicol}

\usepackage{enumitem}
\usepackage{graphicx,subcaption}  
\usepackage{array}

\allowdisplaybreaks[4]

\newcommand{\brab}[1]{\big(#1\big)}
\newcommand{\braB}[1]{\Big(#1\Big)}
\newcommand{\kbrab}[1]{\big[#1\big]}
\newcommand{\kbraB}[1]{\Big[#1\Big]}
\newcommand{\absB}[1]{\Big|#1\Big|}
\newcommand{\myvec}[1]{\boldsymbol{#1}}

\headers{Energy dissipation of IMEX-LMMs}{C. Quan, H. Wang, X. Wang, and C. Xu}

\title{Energy Dissipation Analysis of Implicit-Explicit Linear Multistep Methods for Gradient Flows Using a Simple Multiplier
\thanks{Submitted.
\funding{The work of C. Quan is supported by National Natural Science Foundation of China  (Grant No. 12271241), Guangdong Basic and Applied Basic Research Foundation (Grant No. 2023B1515020030), and Shenzhen Science and Technology Innovation Program (Grant No. JCYJ20230807092402004).  The work of X. Wang is supported by National Natural Science Foundation of China (Grant No. 12501551). The work of C. Xu is supported by National Natural Science Foundation of China (Grant No. 12371408).
}}
}

\author{Chaoyu Quan\thanks{School of Science and Engineering, The Chinese University of Hong Kong (Shenzhen), Shenzhen 518172, China; Shenzhen International Center for Industrial and Applied Mathematics, Shenzhen Research Institute of Big Data, Shenzhen, 518000, China 
  (\email{quanchaoyu@cuhk.edu.cn}).}
\and Huaijin Wang\thanks{School of Mathematical Sciences and Fujian Provincial Key Laboratory of Mathematical Modeling and High Performance Scientific Computing, Xiamen University, Xiamen 361005, China
  (\email{wanghuaijin@stu.xmu.edu.cn}).}
\and Xuping Wang\thanks{School of Science and Engineering, The Chinese University of Hong Kong (Shenzhen), Shenzhen 518172, China 
  (\email{wangxuping@cuhk.edu.cn}).}
\and Chuanju Xu\thanks{School of Mathematical Sciences and Fujian Provincial Key Laboratory of Mathematical Modeling and High Performance Scientific Computing, Xiamen University, Xiamen 361005, China
  (\email{cjxu@xmu.edu.cn}).}
}
\begin{document}

\sloppy
\maketitle
\begin{abstract}
This paper proposes a theoretical framework for establishing the energy dissipation of general implicit-explicit linear multistep methods (IMEX-LMMs) for gradient flows, by constructing a dissipative modified energy consisting of the original energy and a non-negative quadratic modification. We first test IMEX-LMMs with a simple multiplier, the first-order time difference of numerical solutions. Then, it is shown that the associated non-negative quadratic modification can be constructed if and only if two generating polynomials (corresponding to the LMM) are positive on $[-1,1]$. Based on this, the modified energy is proved to decay over time under a mild time-step restriction depending on the lower bounds of the associated generating polynomials. As a consequence, the energy dissipation of the well-known backward differentiation formula methods up to fifth order can be obtained straightforwardly. Furthermore, we construct for the first time (to the best of our knowledge) a sixth-order energy-dissipative IMEX-LMM and also prove the sixth-order barrier of energy-dissipative IMEX-LMMs when testing the simple multiplier. Some numerical experiments are conducted to verify our theoretical results.
\end{abstract}



\section{Introduction}

In this work, we consider the following gradient flow with periodic boundary condition:
\begin{equation}\label{model:phase-field-eq}
    u_{t} = \mathcal{M} \kbrab{\mathcal{L}u + f(u)}, \quad u(0,\boldsymbol{x}) = u^{0}(\boldsymbol{x}) \quad \text{for } (\boldsymbol{x},t) \in \Omega\times (0,T],
\end{equation}
where $\Omega=(-\pi, \pi)^{d} \subset \mathbb{R}^{d} ~(d = 1,2,3)$ is the domain, $T > 0$ is the final time, $u^0$ is the initial value, $\mathcal{M}$ is a self-adjoint, linear, negative definite, and invertible operator, $\mathcal{L}$ is a self-adjoint, linear, and positive semi-definite (PSD) operator, and there exists a non-negative potential function $F$ such that $f(u)=F'(u)$.
Then, the problem \eqref{model:phase-field-eq} can be reformulated as
\begin{equation}\label{eq:gradient-flow}
    u_{t} = \mathcal{M} \frac{\delta E}{\delta u} \quad\text{with}\quad E[u] \coloneqq \int_{\Omega} \kbraB{\frac{1}{2} u \mathcal{L} u + F(u)} \mathrm{d} \boldsymbol{x}.
\end{equation}
The dynamics satisfies the following energy dissipation law:
\begin{equation}\label{ieq:energy-dissipation-law}
    \frac{\mathrm{d}E}{\mathrm{d}t} = \braB{\frac{\delta E}{\delta u}, \frac{\partial u}{\partial t}} = \braB{\mathcal{M}^{-1} \frac{\partial u}{\partial t}, \frac{\partial u}{\partial t}} \leq 0,
\end{equation}
where $(\cdot, \cdot)$ denotes the $L^{2}$ inner product defined by $(f, g) \coloneqq \int_{\Omega} f(\boldsymbol{x}) g(\boldsymbol{x}) \mathrm{d} \boldsymbol{x}$, and $\|\cdot\|$ denotes the induced $L^2$ norm. Notably, by choosing different operators $\mathcal{M}$, $\mathcal{L}$ and $f(u)$, different gradient flows are derived satisfying the energy dissipation property \eqref{ieq:energy-dissipation-law}. One can refer to \cref{sec:preliminaries} for rigorous settings. Currently, the energy dissipation law has been a crucial criterion in designing numerical methods for gradient flows. For example, some interesting one-step methods have been developed to preserve the energy dissipation law, including the exponential time differencing Runge--Kutta (RK) method \cite{FuYang-2022JCP,FuShenYang-2025ETDRK,LiQiaoWangZheng-2026PFC}, the integrating factor RK method \cite{LiaoWangWen-2024CIFRK}, the implicit-explicit RK method \cite{FuTangYang-2024IERK}, and the operator splitting method \cite{LiQuanXu-2022splitting}.


Linear multistep methods (LMMs), such as the $k$-step backward differentiation formula (BDF$k$) methods, are widely used for time integration of stiff differential equations due to their efficiency, such as the advection-diffusion equation \cite{ARW-1995LMM} and the Navier--Stokes equation \cite{ 2017IMEX, 2019IMEX}.
The second Dahlquist barrier indicates that no A-stable LMM can exceed second-order accuracy \cite{Dahlquist-1963-order-barrier}. For BDF1 and BDF2 methods, their energy stability and convergence for parabolic equations have been established via A-stability \cite{book-Thomee2006}.
Analyzing the stability of higher-order or non-A-stable LMMs has been an important issue in the past decades.


In Dahlquist's G-stability theory \cite{Dahlquist-1978-G-stability,baiocchi1989equivalence}, an important step is to find some coefficients $\mu_{i}$ and a PSD matrix $G = (g_{ij})_{i,j=1}^{k-1}$ satisfying the following decomposition:
\begin{equation}\label{eq:structure-general}
\begin{aligned}
    \braB{\sum_{i=0}^{k-1} s_{i} v^{n+1-i}, \sum_{i=0}^{k-1} \mu_{i} v^{n+1-i}} \geq & \sum_{i, j = 0}^{k-2} g_{i+1, j+1} (v^{n+1-i}, v^{n+1-j}) \\
    & - \sum_{i, j = 0}^{k-2} g_{i+1, j+1} (v^{n-i}, v^{n-j})
\end{aligned}
\end{equation}
for any sequence $\{v^{i}\}$ and a given tuple of coefficients $\boldsymbol{s}=[s_0,\ldots,s_{k-1}]^\top\in\mathbb{R}^k$, where
$\sum_{i=0}^{k-1} \mu_{i} v^{i}$ is the multiplier. 
Nevanlinna and Odeh developed a multiplier technique \cite{1981-NO-multiplier}, which was originally designed to study the stability of LMMs for contractive nonlinear ordinary differential equations and extended the analysis beyond the limitations of classical energy methods, particularly for non-A-stable LMMs. This was first generalized to the analysis of time discretization for partial differential equations in \cite{2013-NO-Lubich}, which was subsequently widely applied to stability analysis of fully implicit, linearly implicit, and implicit-explicit BDF methods up to fifth order \cite{2015SINUM-NO-Akrivis,2015NM-NO-AkrivisLubich}. More generally, Akrivis et al. \cite{2016MCOM-NO-Akrivis} constructed optimal multipliers for the BDF3 and BDF5 methods, respectively, and proved that the Nevanlinna--Odeh multiplier is optimal for the BDF4 method. Due to the nonexistence of Nevanlinna--Odeh multiplier for BDF6 method, relaxed multiplier conditions were introduced for stability analysis \cite{2021SINUM-NO-Akrivis-BDF6}. Recently, since the BDF7 method was unstable, a weighted and shifted seven-step BDF method was proposed for parabolic equations using the multiplier technique \cite{2025IMA-NO-Akrivis-BDF7}. The multiplier technique has potential to be applied to other LMMs, such as the improved Gear method \cite{1991WBDF}, generalized BDF method \cite{HuangShen-2024SINUM, HuangShen-2025MCOM} and modified BDF method \cite{2003MBDF}.

To establish the energy dissipation of IMEX-LMMs for gradient flows, one commonly uses the simple multiplier: $\mu_0=1$ and $\mu_i=0$ for $1 \leq i \leq k-1$ in \eqref{eq:structure-general}.
Furthermore, an additional positive quadratic term shall be used to control the explicit treatment of nonlinear term, which gives the following quadratic decomposition similar to \eqref{eq:structure-general}:
\begin{equation}\label{eq:structure2}
\begin{aligned}
    \Big(\sum_{i=0}^{k-1} s_{i} v^{n+1-i},v^{n+1}\Big) &\geq \sum_{i, j = 0}^{k-2} g_{i+1, j+1} (v^{n+1-i}, v^{n+1-j}) \\
    &\qquad\qquad - \sum_{i, j = 0}^{k-2} g_{i+1, j+1} (v^{n-i}, v^{n-j}) + \gamma \|v^{n+1}\|^2
\end{aligned}
\end{equation}
with a constant $\gamma>0$, for any sequence $\{v^{i}\}$. Using this decomposition and testing the scheme with $v^{n+1} = u^{n+1}-u^{n}$ (the first-order time difference of the numerical solution), the energy dissipation of the first- and second-order IMEX-BDF methods for gradient flows has been well-studied \cite{ChenShen1998,XuTang2006SINUM}. For the BDF3 method, the explicit form of $g_{i,j}$ can be derived and a modified energy can subsequently be constructed  \cite{hao2020third}. Further, the decomposition  \eqref{eq:structure2} for the BDF methods up to fifth order can be constructed, while the BDF6 method does not have such a decomposition if testing the scheme with $v^{n+1} =u^{n+1}-u^{n}$ \cite{LiaoTangZhou-2022CSIAM,LiaoKang-2024IMA-PFC}.

One question is how to overcome the fifth-order barrier under the aforementioned simple multiplier $\mu_0=1$ and $\mu_i=0$ for $1 \leq i \leq k-1$, for general LMMs. In this work, we propose a unified framework to establish the energy dissipation of IMEX-LMMs for gradient flows. Moreover, we construct a sixth-order energy-dissipative IMEX-LMM and demonstrate the nonexistence  of a seventh-order energy-dissipative IMEX-LMM under the simple multiplier. Our main contributions include:
\begin{enumerate}
    \item[(1)] By taking $u^{n+1}-u^{n}$ as the test function, we first provide two positive-definiteness conditions of the form \eqref{eq:structure2} that are sufficient to establish the energy dissipation of IMEX-LMMs in \cref{theorem: LMM for linear parabolic problem}.
    We then show that there exists a PSD matrix $G\in\mathbb{R}^{(k-1)\times(k-1)}$ satisfying the decomposition \eqref{eq:structure2} if and only if the associated real-valued generating polynomial is positive on $[-1,1]$ (see \cref{thm:lr:equi}).
    As a consequence, the dissipative modified energy can be constructed under mild time-step restrictions if two real-valued generating polynomials associated with the IMEX-LMM coefficients are positive on $[-1,1]$ (see \cref{thm:main-result}).

    \item[(2)] Considering the positivity of the real-valued generating polynomials and the order conditions, we solve a feasibility problem to derive, for the first time (to the best of our knowledge), a sixth-order energy-dissipative IMEX-LMM. Meanwhile, through Farkas' Lemma, we rigorously prove the nonexistence of seventh-order energy-dissipative IMEX-LMMs for the multiplier $v^{n+1} = u^{n+1}-u^n$, i.e., the sixth-order barrier.

\end{enumerate}

The rest of the paper is organized as follows. \Cref{sec:preliminaries} introduces some preliminaries. In \cref{sec:framework and ROR}, we propose positive-definiteness conditions sufficient for the energy dissipation of IMEX-LMMs and provide equivalent statements to these conditions.
In \cref{sec:ConstructionLMM}, this analysis framework is applied to the well-known IMEX-BDF$k$ methods, to construct a sixth-order energy-dissipative IMEX-LMM, and to prove the nonexistence of seventh-order IMEX-LMMs satisfying our framework. Some numerical experiments are presented to verify the convergence order and the energy dissipation law in \cref{sec:Numerical experiments}, followed by a conclusion in \cref{sec:conclusion}.

\section{Preliminaries}\label{sec:preliminaries}
This section collects some preliminaries required for the subsequent analysis; see, e.g., \cite{brezis2011functional,book-solving-ODE2}.

\subsection{Abstract settings}

Let $H$ be a Hilbert space with inner product $(\cdot,\cdot)$ and induced norm $\|\cdot\|$. 
Assume $S$ and $V$ are Hilbert spaces such that $S \hookrightarrow V \hookrightarrow H$ with continuous and compact embeddings. Let $V^\prime$ be the dual of $V$ and denote by $\langle\cdot,\cdot\rangle_{V',V}$ the duality pairing. Identifying $H$ with its dual $H^\prime$, we have the Gelfand triple $V \subset H \subset V^\prime$.

Assume that $\mathcal{L}: S \to V$ is a linear, self-adjoint, and PSD operator:
\[
(\mathcal{L} u, v) = (u, \mathcal{L} v), \quad (\mathcal{L} v, v) \geq 0, \quad \forall u,v\in S.
\]
The associated semi-norm is defined by $\| \mathcal{L}^{1/2} u\| \coloneqq (\mathcal{L} u, u)^{1/2}$.

Assume that $\mathcal{M}:V\to V'$ is a linear self-adjoint operator:
\[
\langle \mathcal{M} u, v\rangle_{V',V} = \langle \mathcal{M} v, u\rangle_{V',V}, \quad \forall u, v\in V,
\]
and that the bilinear form $a_{\mathcal M}(u,v):=-\langle \mathcal M u,v\rangle_{V',V}$ 
is continuous and coercive; that is, there exist constants
$C_{\mathcal M},c_{\mathcal M}>0$ such that
\[
|\langle \mathcal{M}u,v\rangle_{V',V}|
\leq C_{\mathcal M}\|u\|_V\|v\|_V,
\quad
-\langle \mathcal{M}v,v\rangle_{V',V}
\geq c_{\mathcal M}\|v\|_V^2,
\quad \forall u, v\in V,
\]
where $\|\cdot\|_V$ denotes the norm in $V$.
%
%
Then, by the Lax--Milgram theorem, the inverse operator $\mathcal{M}^{-1}: V^\prime \to V$ is well-defined, linear, and bounded; moreover, it is self-adjoint and negative definite.
For $v \in H$, one can define the $(-\mathcal{M})^{-1}$-induced norm as $\|(-\mathcal{M})^{-1/2} v\| \coloneqq (v, -\mathcal{M}^{-1}v)^{1/2}$.

Assume that $f:S\to V$ is a (possibly nonlinear) operator satisfying the global Lipschitz condition with constant $\ell_f>0$:
\begin{equation}
\label{eq:lip-f}
\|f(u)-f(v)\| \leq \ell_f \|u-v\|,\quad \forall u,v \in S.
\end{equation}
This implies that the potential function $F$ satisfies the following inequality:
\begin{equation}
\label{eq:lip-F}
(F(u) - F(v), 1) \leq (f(v), u-v) + \frac{\ell_f}{2} \| u-v \|^2,\quad \forall u,v\in S.
\end{equation}

In addition, we assume that there exist two constants $\zeta > 0$ and $0 < \eta \leq 1$ such that the following inequality holds:
\begin{equation}
\label{eq:l2-control}
\|v\| \leq \zeta \| (-\mathcal{M})^{-1/2} v\|^\eta  \| \mathcal{L}^{1/2} v \|^{1-\eta},\quad \forall v\in S.
\end{equation}

In the following examples, the assumptions are understood in the indicated spaces or subspaces so that the inequality \eqref{eq:l2-control} holds.
\begin{itemize}
    \item Allen--Cahn equation: $\mathcal{M}=-\mathcal{I}$, $\mathcal{L} = -\varepsilon^2 \Delta$,  $f(u) = u^3-u$, $H=L^2(\Omega)$, $V=L^2(\Omega)$, and $S = H^2(\Omega)$. The inequality \eqref{eq:l2-control} holds for $\zeta=1$ and $\eta=1$.
    \item Cahn--Hilliard equation: $\mathcal{M} = \Delta$, $\mathcal{L} = -\varepsilon^2 \Delta$, $f(u) = u^3-u$, $H=\dot{L}^2(\Omega)$, $V=\dot{H}^1(\Omega)$, and $S=\dot{H}^3(\Omega)$. The inequality \eqref{eq:l2-control} holds for $\zeta=\varepsilon^{-1/2}$ and $\eta=1/2$.
    \item Phase field crystal (PFC) equation: $\mathcal{M} = \Delta$, $\mathcal{L} = (\mathcal{I}+ \Delta)^2 + \mathcal{I}$, $f(u) = u^3-(\varepsilon+1) u$, $H=\dot{L}^2(\Omega)$, $V=\dot{H}^1(\Omega)$, and $S=\dot{H}^5(\Omega)$. The inequality \eqref{eq:l2-control} holds for $\zeta= (2\sqrt{2}-2)^{-1/4}$ and $\eta=1/2$.
\end{itemize}
Here, the Sobolev spaces are actually periodic Sobolev spaces as the periodic boundary condition is equipped with, and for a Hilbert space $X$, $\dot{X} = \{v\in X: \int_{\Omega} v \mathrm{d} \boldsymbol{x} = 0\}$ denotes the subspace of functions with zero mean, which is a common requirement for mass-conserving gradient flows.

\begin{remark}
\label{rem:truncationTech}
The Lipschitz continuity assumption for the nonlinear term is widely used in the analysis of numerical methods for gradient flows. For potentials that do not satisfy this condition globally, a common remedy in numerical simulations is to use truncation techniques; see, e.g., \cite{ShenYang-2010-Lipschitz,LiQiaoWang-2021MCOM-truncation,FuTangYang-2024IERK}. More precisely, one can replace $f$ by a globally Lipschitz continuous truncated function $f_R$, satisfying $f_R(s)=f(s)$ for $|s|\le R$ and that $|f_R(s)|$ is bounded for $|s|>R$. 
In practical numerical implementations, $R$ is usually chosen to be sufficiently large such that the numerical solution remains in $[-R,R]$, implying that $f_R$ and $f$ actually agree along the solution trajectory.
For example, we chose $R=2$ in \cref{ex:energy-PFC} for the PFC model where the absolute value of the numerical solution is observed to not exceed $R$.
\end{remark}

\subsection{Formulation of IMEX-LMMs}
Let $t_n = n\tau$, $n=0,\ldots,N$, be a uniform partition of $[0,T]$ with step size $\tau = T/N$. Given the starting approximations $\{u^i\}_{i=0}^{k-1} \in S$,  the $k$-step IMEX-LMM for \eqref{model:phase-field-eq} is
\begin{equation}\label{eq:model:lmm2}
    \sum_{i=0}^k A_i^{(k)} u^{n+1-i} = \tau \mathcal{M} \kbraB{\sum_{i=0}^k B_i^{(k)} \mathcal{L} u^{n+1-i} + \sum_{i=1}^k \hat{B}_i^{(k)} f(u^{n+1-i})}
\end{equation}
for $k-1\leq n\leq N-1,$
where the coefficients satisfy the $k$th-order conditions:
\begin{equation}
\label{constraint: order conditions of LMM}
\left\{
\begin{aligned}
& \sum_{i=0}^k A_i^{(k)} = 0, \\
& \sum_{i=0}^k A_i^{(k)} (-i)^{m+1} = (m+1) \sum_{i=0}^k B_i^{(k)} (-i)^{m}, \quad 0\leq m\leq k-1,\\
& \sum_{i=0}^k A_i^{(k)} (-i)^{m+1} = (m+1) \sum_{i=1}^k \hat{B}_i^{(k)} (-i)^{m}, \quad 0\leq m\leq k-1.\\
\end{aligned}
\right .
\end{equation}
In practice, the starting approximations $u^0,\ldots,u^{k-1}$ can be generated by a sufficiently accurate one-step method, for example the Gauss collocation RK method with small time steps. 
The energy dissipation analysis in the following content is for $n\geq k-1$.

To simplify the analysis, we further impose the following normalization condition:
\begin{equation}
\label{condition: normalization}
 \sum_{i=0}^k A_i^{(k)} (-i) =  \sum_{i=0}^k B_i^{(k)} = \sum_{i=1}^k \hat{B}_i^{(k)} = 1,
\end{equation}
which ensures that the discrete operator $\frac{1}{\tau} \sum_{i=0}^{k} A_i^{(k)} u(t_{n+1-i})$ is a first-order approximation of $u_t(t_{n+1})$.
By inverting $\tau \mathcal M$, \eqref{eq:model:lmm2} can be reformulated as
\begin{equation}\label{scheme: LMM-difference form}
\begin{aligned}
&\frac{1}{\tau} \sum_{i=0}^{k-1} a_{i}^{(k)} \mathcal{M}^{-1}\brab{\delta u^{n+1-i}}
- \frac{1}{2} \mathcal{L} (u^{n+1} + u^{n})
- \sum_{i=0}^{k-1} b_{i}^{(k)} \mathcal{L} \delta u^{n+1-i} \\
&\qquad\qquad\qquad= f(u^n)
+ \sum_{i=1}^{k-1} \hat{b}_i^{(k)} \delta f(u^{n+1-i}) \quad\text{for $k-1\leq n\leq N-1,$}
\end{aligned}
\end{equation}
where $\delta u^{n} \coloneqq u^{n} - u^{n-1}$ denotes the  first-order time difference and the reformulated coefficients are given by the cumulative sums:\begin{equation}
\label{def: coefficients a b c}
\left\{
\begin{aligned}
& a_i^{(k)} = \sum_{j=0}^i A_j^{(k)},&& 0\leq i\leq k-1 ,\\
& b_i^{(k)} = \sum_{j=0}^i B_j^{(k)} - 1 + \frac{1}{2}\delta_{i,0},&& 0\leq i \leq k-1, \\
& \hat{b}_i^{(k)} = \sum_{j=1}^i \hat{B}_j^{(k)} - 1 ,&& 1\leq i\leq k-1,
\end{aligned}
\right .
\end{equation}
with $\delta_{i,0}$ being the Kronecker delta. For simplicity, we define the coefficient vectors:
\begin{equation}
\label{eq:coef-vecs}
\boldsymbol{a}^{(k)} = [a_0^{(k)},\ldots,a_{k-1}^{(k)}]^\top,\  \boldsymbol{b}^{(k)}=[b_0^{(k)},\ldots,b_{k-1}^{(k)}]^\top, \ 
\text{and} \ 
\hat{\boldsymbol{b}}^{(k)}=[\hat{b}_1^{(k)},\ldots,\hat{b}_{k-1}^{(k)},0]^\top.
\end{equation}

\section{Energy dissipation preservation of IMEX-LMMs}\label{sec:framework and ROR}


In the following analysis, we establish the energy dissipation law for the reformulation \eqref{scheme: LMM-difference form} of the IMEX-LMMs. {To better present our result, we define the real-valued generating function} as 
\begin{equation}\label{def: polynomial M(theta)}
    M(\theta;\boldsymbol{s}) \coloneqq \sum_{m=0}^{k-1} s_{m} \cos(m \theta),\quad \theta \in [0, \pi],
\end{equation}
for a given coefficient vector $\boldsymbol{s} = [s_0,\ldots,s_{k-1}]^\top$. 
The vector $\boldsymbol{s}$ can be taken as either $\boldsymbol{a}^{(k)}$ or $\boldsymbol{b}^{(k)}$ defined in \eqref{eq:coef-vecs}.
By the variable transformation $\theta  =  \arccos x$, $x\in [-1,1]$, \eqref{def: polynomial M(theta)} can be transformed to a real-valued generating polynomial:
\begin{equation}
\label{eq:def:chb}
T(x;\boldsymbol{s}) \coloneqq M(\arccos x;\boldsymbol{s}),\quad -1\leq x \leq 1,
\end{equation}
which can be expressed as the Chebyshev expansion:
\begin{equation}
\label{eq:t-chebs}
T(x) = \sum_{m=0}^{k-1} s_m T_m(x),\quad -1\leq x \leq 1,
\end{equation}
where $T_m(x) = \cos(m\theta)$, $x = \cos\theta \in [-1, 1]$, denotes the $m$th Chebyshev polynomial of the first kind. It is well known that these polynomials satisfy the  three-term recurrence relation:
\[
T_0(x) = 1,\quad T_1(x) = x,\quad T_{m+1} = 2xT_m(x) - T_{m-1} (x), \quad m=1,2,\ldots.
\]
Then, it is clear that
\begin{equation}
\label{eq:same-min}
\min_{\theta \in [0,\pi]} M(\theta;\boldsymbol{s}) = \min_{ x\in [-1,1] } T(x;\boldsymbol{s}).
\end{equation} 
For a nonsymmetric matrix $G$, we use the convention throughout this paper that $G$ is called 
PSD if its symmetric part $(G+G^\top)/2$
is PSD. That is, $\boldsymbol{x}^{\top}G\boldsymbol{x}\geq 0$ for all $\boldsymbol{x}$.

Throughout this section, for two vectors of  functions $\boldsymbol{w} = [w^1,\ldots,w^m]^\top$ and $\boldsymbol{v}=[v^1,\ldots,v^m]^\top$, with $w^i,v^i\in H$, we define the product-space inner product by
\[
(\boldsymbol{w},\boldsymbol{v})
\coloneqq \sum_{i=1}^m (w^i,v^i).
\]
More generally, for a matrix $G=(g_{ij})\in\mathbb R^{m\times m}$, we define the $G$-weighted bilinear form by
\[
(\boldsymbol{w}, \boldsymbol{v})_G\coloneqq \sum_{i,j=1}^m g_{ij}(w^i,v^j).
\]
If $\mathscr{T}$ is a linear operator, typically either $\mathcal{L}$ or
$\mathcal{M}^{-1}$, then $\mathscr{T}\boldsymbol v$ is understood componentwise, namely $\mathscr{T}\boldsymbol v
=[\mathscr{T}v^1,\ldots,\mathscr{T}v^m]^\top$.

In addition, we shall define the following modified energy:
\begin{equation}
    \label{eq:EG}
    \boxed{
        E^{n}_{G} \coloneqq E[u^n] - \frac{1}{\tau}
        (\boldsymbol{v}_n,\mathcal{M}^{-1} \boldsymbol{v}_{n})_{G_a} 
        + 
        (\boldsymbol{v}_n,\mathcal{L} \boldsymbol{v}_{n})_{G_b}        
        + \ell_f  \sum_{i=1}^{k-1} \hat{c}_i^{(k)} \| \delta u^{n+1-i} \|^2
        }
    \end{equation}
for given upper triangular PSD matrices $G_a, G_b\in \mathbb{R}^{(k-1)\times (k-1)}$. 
Here, 
\begin{equation}
\label{eq:eng-vn}
\boldsymbol{v}_{n} = [\delta u^{n}, \delta u^{n-1}, \dots, \delta u^{n+2-k}]^\top,
\end{equation}
$\ell_f$ is the Lipschitz constant in \eqref{eq:lip-f}, {the inner product of two  vector functions is denoted by $\boldsymbol{u}^\top \boldsymbol{v} = \sum_i (u^i,v^i)$}, and
\begin{equation}\label{def:ci-hat}
    \hat{c}_i^{(k)} = \frac{1}{2} \sum_{j=i}^{k-1} | \hat{b}_j^{(k)}| \quad\text{for $i=1,\ldots,k-1$.}
\end{equation}

{We are now ready to state and prove our main result on the modified energy dissipation of the IMEX-LMMs for gradient flows.}

\begin{theorem}
\label{thm:main-result}
{Suppose that for the gradient flow \eqref{model:phase-field-eq} with Lipschitz continuous nonlinearity satisfying \eqref{eq:lip-f}, there exist $\zeta>0$ and $0<\eta\leq 1$ such that the inequality \eqref{eq:l2-control} holds.} 
Consider the $k$-step IMEX-LMM \eqref{scheme: LMM-difference form} with coefficient vectors $\boldsymbol{a}^{(k)}$, $\boldsymbol{b}^{(k)}$ and $\hat{\boldsymbol{b}}^{(k)}$ defined in \eqref{eq:coef-vecs}. 
  If 
\begin{equation}
\label{eq:lr:es}
\min_{\theta\in [0,\pi] } M(\theta; \boldsymbol{a}^{(k)}) \geq \alpha >0 \quad \text{and} \quad \min_{\theta\in[0,\pi]} M(\theta; \boldsymbol{b}^{(k)} ) \geq \beta > 0,
\end{equation}
or equivalently,
\begin{equation}
\label{eq:lr:es2}
\min_{x\in [-1,1] } T(x; \boldsymbol{a}^{(k)}) \geq \alpha > 0 \quad \text{and} \quad \min_{x\in[-1,1]} T(x; \boldsymbol{b}^{(k)} ) \geq \beta > 0,
\end{equation}
then there exist two upper triangular PSD matrices $G_a$ and $G_b$ such that the modified energy \eqref{eq:EG} satisfies the dissipation property $E_G^{n+1}\leq E_G^n$, under the time-step restriction:
\begin{equation}
\label{eq:tau-max}
0<\tau \leq \tau_{\max}, \quad \tau_{\max} \coloneqq 
\frac{\alpha \beta^{\bar{\eta}}}{| \frac{\ell_f}{2} + 2\ell_f  \hat{c}_1^{(k)} |^{1+\bar{\eta}} \eta (1-\eta)^{\bar{\eta}} \zeta^{2+2\bar{\eta}} },
\end{equation}
where $\ell_f$ is the Lipschitz constant in \eqref{eq:lip-f}, $\hat{c}_1^{(k)}$ is given by \eqref{def:ci-hat}, and $\bar{\eta} = \frac{1-\eta}{\eta}$. Here, we make the convention that $0^0=1$ for $\eta=1$.
\end{theorem}

\begin{proof}
    {The proof will be completed by combining \cref{theorem: LMM for linear parabolic problem} on positive-definiteness conditions for energy dissipation and \cref{thm:lr:equi} on the equivalent statements of these conditions.}
\end{proof}

In the case of $\eta = 1$, $\tau_{\max}$ defined in \eqref{eq:tau-max} is independent of $\beta$, where $(1-\eta)^{\bar{\eta}}=0^0=1$. As a consequence, the condition $\beta> 0$ in \eqref{eq:lr:es} and \eqref{eq:lr:es2} can be relaxed to $\beta \geq 0$. However, for the sake of brevity, we do not treat this as a separate case.

\subsection{Positive-definiteness conditions for energy dissipation}
\label{subsec:edf}
In this subsection, we propose two positive-definiteness conditions that are sufficient to establish the energy dissipation law of IMEX-LMMs. Henceforth, $\boldsymbol{0}_{m\times n}$ denotes the zero matrix in $\mathbb{R}^{m\times n}$. When the dimension of $\boldsymbol{0}_{m\times n}$ is clear in the following content, we omit its subscripts.

\begin{lemma}\label{theorem: LMM for linear parabolic problem}
{Suppose that for the gradient flow \eqref{model:phase-field-eq}  with Lipschitz continuous nonlinearity satisfying \eqref{eq:lip-f}, there exist $\zeta>0$ and $0<\eta\leq 1$ such that the inequality \eqref{eq:l2-control} holds.}
Consider the $k$-step IMEX-LMM \eqref{scheme: LMM-difference form} with coefficient vectors $\boldsymbol{a}^{(k)}$, $\boldsymbol{b}^{(k)}$ and $\hat{\boldsymbol{b}}^{(k)}$ defined in \eqref{eq:coef-vecs}. 
If there exist upper triangular PSD matrices $G_{a}, G_{b} \in \mathbb{R}^{(k-1)\times(k-1)}$ and constants $\alpha, \beta > 0$ satisfying the following positive-definiteness conditions:
    \begin{equation}\label{assumption: Ua and Ub}
        \boldsymbol{x}^\top U_{a}^{(k)} \boldsymbol{x} \geq \alpha x_{1}^{2}, \quad \boldsymbol{x}^\top U_{b}^{(k)} \boldsymbol{x} \geq \beta x_{1}^{2}, \quad \forall \boldsymbol{x}= [x_{1}, x_{2}, \dots, x_{k}]^\top \in \mathbb{R}^{k},
    \end{equation}
    where
    \begin{equation}\label{def: matrix Ua and Ub}
        U_{a}^{(k)} = \boldsymbol{e}_1 (\boldsymbol{a}^{(k)})^\top - \begin{bmatrix}
            G_{a} & \boldsymbol{0} \\
            \boldsymbol{0}^\top & 0
        \end{bmatrix} + \begin{bmatrix}
            0 & \boldsymbol{0}^\top \\
            \boldsymbol{0} & G_{a}
        \end{bmatrix}, \quad 
        U_{b}^{(k)} = \boldsymbol{e}_1 (\boldsymbol{b}^{(k)})^\top - \begin{bmatrix}
            G_{b} & \boldsymbol{0} \\
            \boldsymbol{0}^\top & 0
        \end{bmatrix} + \begin{bmatrix}
            0 & \boldsymbol{0}^\top \\
            \boldsymbol{0} & G_{b}
        \end{bmatrix},
    \end{equation}
    with $\boldsymbol{e}_1 = [1,0,\dots,0]^\top \in \mathbb{R}^k$,  then the modified energy \eqref{eq:EG} satisfies the dissipation property $E^{n+1}_G \leq E_G^{n}$ under the time-step restriction \eqref{eq:tau-max}.
     
\end{lemma}

\begin{proof}
Let $\boldsymbol{w}_{n} = [\delta u^{n}, \delta u^{n-1}, \dots, \delta u^{n+1-k} ]^\top$. Taking the inner product of \eqref{scheme: LMM-difference form} with $\delta u^{n+1}$, the derived left-hand side is
\begin{equation}
\label{eq:thm1:temp1}
\begin{aligned}
    \text{LHS} =&\  \frac{1}{\tau} 
    (\boldsymbol{w}_{n+1},\mathcal{M}^{-1} \boldsymbol{w}_{n+1})_{\boldsymbol{e}_1 (\boldsymbol{a}^{(k)})^\top}
    -\frac{1}{2} \| \mathcal{L}^{1/2} u^{n+1} \|^2 + \frac{1}{2} \| \mathcal{L}^{1/2} u^n \|^2  \\
    & - ( \boldsymbol{w}_{n+1}, \mathcal{L} \boldsymbol{w}_{n+1})_{\boldsymbol{e}_1 (\boldsymbol{b}^{(k)})^\top}.
\end{aligned}
\end{equation}
Using the definitions in \eqref{def: matrix Ua and Ub}, we observe that
\begin{equation}
\label{eq:thm1:temp2}
\begin{aligned}
    (\boldsymbol{w}_{n+1},\mathcal{M}^{-1} \boldsymbol{w}_{n+1})_{\boldsymbol{e}_1 (\boldsymbol{a}^{(k)})^\top} =&\   (\boldsymbol{w}_{n+1},\mathcal{M}^{-1} \boldsymbol{w}_{n+1})_{U_a^{(k)}} \\
    & + (\boldsymbol{v}_{n+1},\mathcal{M}^{-1} \boldsymbol{v}_{n+1})_{G_a}
     - (\boldsymbol{v}_{n},\mathcal{M}^{-1} \boldsymbol{v}_{n})_{G_a},
\end{aligned}
\end{equation}
and similarly,
\begin{equation}
\label{eq:thm1:temp3}
( \boldsymbol{w}_{n+1}, \mathcal{L} \boldsymbol{w}_{n+1})_{\boldsymbol{e}_1 (\boldsymbol{b}^{(k)})^\top}
= ( \boldsymbol{w}_{n+1}, \mathcal{L} \boldsymbol{w}_{n+1})_{U_b^{(k)}} +  ( \boldsymbol{v}_{n+1}, \mathcal{L} \boldsymbol{v}_{n+1})_{G_b} - ( \boldsymbol{v}_{n}, \mathcal{L} \boldsymbol{v}_{n})_{G_b},
\end{equation}
where $\boldsymbol{v}_n$ is given by \eqref{eq:eng-vn} that is actually the first $k-1$ entries of $\boldsymbol{w}_n$.
The derived right-hand side is
\begin{equation}
\label{eq:thm1:temp4}
\text{RHS}
=  \brab{f (u^n) , \delta u^{n+1}}
+ \sum_{i=1}^{k-1} \hat{b}_i^{(k)} \brab{\delta f (u^{n+1-i}), \delta u^{n+1}}.
\end{equation}
It is known from \eqref{eq:lip-F} that
\begin{equation}
\label{eq:thm1:temp5}
\brab{f (u^n) , \delta u^{n+1}} 
\geq \brab{F (u^{n+1}) - F (u^{n}), 1} - \frac{\ell_f}{2} \|\delta u^{n+1}\|^2 .
\end{equation}
By the Cauchy--Schwarz inequality, we have
\begin{align}
\sum_{i=1}^{k-1} \hat{b}_i^{(k)} &
\brab{\delta f (u^{n+1-i}), \delta u^{n+1}}
\ge
- \frac{\ell_f}{2} \sum_{i=1}^{k-1} |\hat{b}_i^{(k)}|
\brab{\|\delta u^{n+1-i} \|^2+ \|\delta u^{n+1}\|^2} \notag \\
&=
- \ell_f \sum_{i=1}^{k-1} \hat{c}_i^{(k)} \| \delta u^{n+1-i}\|^2 
+ \ell_f \sum_{i=1}^{k-1} \hat{c}_i^{(k)} \| \delta u^{n+2-i}\|^2
- 2\ell_f\hat{c}_1^{(k)} \|\delta u^{n+1}\|^2,
\label{eq:thm1:temp6}
\end{align}
where the definition of {$\hat{c}_i^{(k)}$} is given by \eqref{def:ci-hat}. 
Substituting \eqref{eq:thm1:temp5} and \eqref{eq:thm1:temp6} into \eqref{eq:thm1:temp4},
we arrive at
\begin{equation}
\label{eq:thm1:temp7}
\begin{aligned}
\text{RHS} \geq 
&\ \brab{F (u^{n+1}) - F (u^{n}), 1}  - \ell_f \sum_{i=1}^{k-1} \hat{c}_i^{(k)} \| \delta u^{n+1-i}\|^2 \\
& + 
 \ell_f \sum_{i=1}^{k-1} \hat{c}_i^{(k)} \| \delta u^{n+2-i}\|^2 -\braB{\frac{\ell_f}{2} +2\ell_f \hat{c}_1^{(k)}} \|\delta u^{n+1}\|^2.
\end{aligned}
\end{equation}
Combining \eqref{eq:thm1:temp7} with \eqref{eq:thm1:temp1}-\eqref{eq:thm1:temp3}, we have
\begin{equation}
\label{eq:thm1:temp8}
\begin{aligned}
&E^{n+1}_G - E^n_G \\
 \leq &\ 
\frac{1}{\tau}
(\boldsymbol{w}_{n+1},\mathcal{M}^{-1} \boldsymbol{w}_{n+1})_{U_a^{(k)}}
 - 
(\boldsymbol{w}_{n+1}, \mathcal{L} \boldsymbol{w}_{n+1})_{U_b^{(k)}}
+\braB{\frac{\ell_f}{2} + 2\ell_f \hat{c}_1^{(k)}} \|\delta u^{n+1}\|^2 \\
 \leq & -\frac{\alpha}{\tau} \| (-\mathcal{M})^{-1/2} \delta u^{n+1}\|^2 - \beta \|\mathcal{L}^{1/2}  \delta u^{n+1} \|^2
+\absB{\frac{\ell_f}{2} + 2\ell_f \hat{c}_1^{(k)}} \|\delta u^{n+1}\|^2. \\
\end{aligned}
\end{equation}

{In the case of $\eta = 1$, \eqref{eq:l2-control} gives $\|\delta u^{n+1}\|\leq \zeta \| (-\mathcal{M})^{-1/2} \delta u^{n+1}\|$. 
It is not difficult to find from \eqref{eq:thm1:temp8} that 
$E_{G}^{n+1}\leq E_G^n$
when $\tau$ satisfies \eqref{eq:tau-max}. 
In the case of $0<\eta<1$, we apply Young's inequality $ab\leq \frac{a^p}{p}+\frac{b^q}{q}$ with $p=\frac{1}{\eta}$ and $q = \frac{1}{1-\eta}$. Replacing $a$ and $b$ by  $\xi^{\eta-1} a^{2\eta}$ and $\xi^{1-\eta}  b^{2-2\eta}$ respectively for any $\xi>0$, yields
\begin{equation}
a^{2\eta} b^{2-2\eta} \leq \eta \xi^{-\frac{1-\eta}{\eta}} a^2 + (1-\eta)\xi b^2.
\label{eq:thm1:temp9}
\end{equation}
Applying \eqref{eq:thm1:temp9} to \eqref{eq:l2-control}, we have}
\[
\begin{aligned}
\|\delta u^{n+1}\|^2
 & \leq  \zeta^2 \| (-\mathcal{M})^{-1/2} \delta u^{n+1} \|^{2\eta} \| \mathcal{L}^{1/2} \delta u^{n+1} \|^{2-2\eta} \\
& \leq  \zeta^2 \kbraB{\eta \xi^{-\frac{1-\eta}{\eta}} \| (-\mathcal{M})^{-1/2} \delta u^{n+1}\|^2
+ (1-\eta) \xi \| \mathcal{L}^{1/2} \delta u^{n+1}\|^2}.
\end{aligned}
\]
To ensure the decay of $E_G^n$, we then impose
\begin{equation}
\label{eq:thm1:temp10}
-\frac{\alpha}{\tau} + \braB{\frac{\ell_f}{2} + 2\ell_f  \hat{c}_1^{(k)}} \zeta^2 \eta \xi^{-\frac{1-\eta}{\eta}} \leq 0 \ \text{and}\
 -\beta + \braB{\frac{\ell_f}{2} + 2\ell_f  \hat{c}_1^{(k)}} \zeta^2 (1-\eta) \xi \leq 0. 
\end{equation}
By taking 
\[
\xi = \frac{\beta}{|\frac{\ell_f}{2} + 2\ell_f \hat{c}_1^{(k)}| \zeta^2 (1-\eta)},
\]
{the second inequality of \eqref{eq:thm1:temp10} holds naturally} and we then obtain the time-step restriction \eqref{eq:tau-max}.
\end{proof}

\begin{remark}
    For the IMEX-BDF2 method, the positive-definiteness conditions of \cref{theorem: LMM for linear parabolic problem} in \eqref{assumption: Ua and Ub} give the  feasible regions of $(G_a,\alpha)$ and $(G_b,\beta)$, as shown in \cref{fig:BDF2-es}.
\end{remark}

\begin{figure}[htbp]
    \centering
    \begin{subfigure}[b]{0.45\textwidth}
        \centering
        \includegraphics[width=\textwidth]{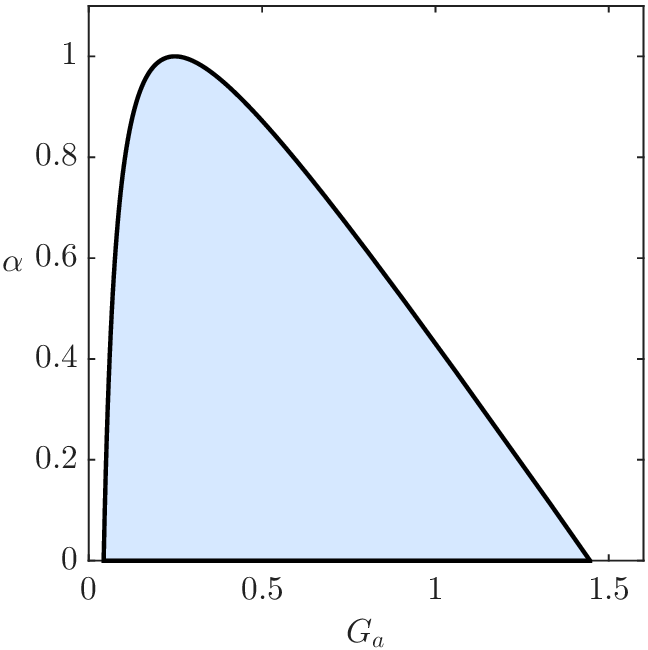}
    \end{subfigure}
    \quad
    \begin{subfigure}[b]{0.45\textwidth}
        \centering
        \includegraphics[width=\textwidth]{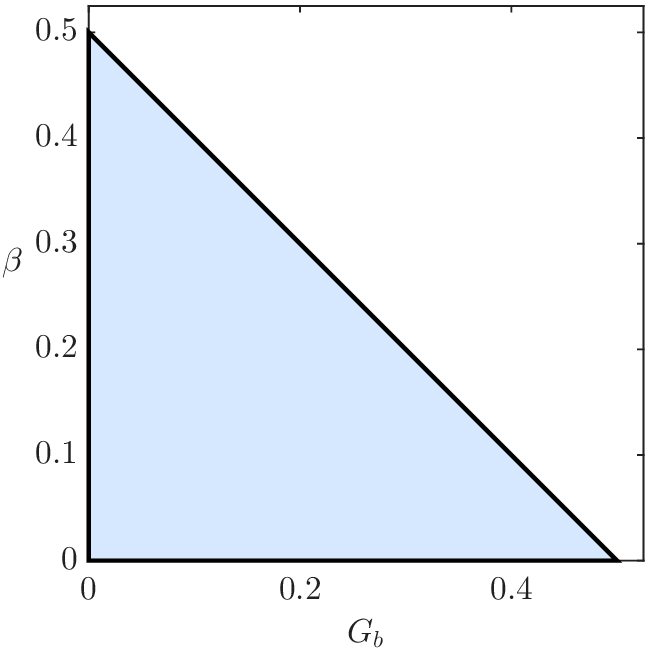}
    \end{subfigure}
    \caption{Feasible regions of $(G_a, \alpha)$ (left) and $(G_b, \beta)$ (right) for the IMEX-BDF2 method under the positive-definiteness conditions of \cref{theorem: LMM for linear parabolic problem} in \eqref{assumption: Ua and Ub} are shaded in light blue.}
    \label{fig:BDF2-es}
\end{figure}

\subsection{Equivalent statements of positive-definite conditions}
\label{subsec:ror}

In this subsection, we present equivalent statements of the positive-definiteness conditions of \cref{theorem: LMM for linear parabolic problem} in \eqref{assumption: Ua and Ub}.
For $k\geq 2$, to study the conditions in \eqref{assumption: Ua and Ub}, for a given vector $\boldsymbol{s}=[s_0,\ldots,s_{k-1}]^\top\in\mathbb{R}^k$, we consider the problem to find a scalar $\gamma > 0$ and an upper triangular PSD matrix $G\in\mathbb{R}^{(k-1)\times(k-1)}$ such that the following positive-definiteness condition holds:
\begin{equation}\label{assumption: U}
  \boldsymbol{x}^\top U \boldsymbol{x} \geq \gamma x_{1}^{2}, \quad \forall \boldsymbol{x}=[x_{1}, x_{2}, \ldots, x_{k}]^{\top} \in \mathbb{R}^{k},
\end{equation}
where the matrix $U$ is defined by
\begin{equation}\label{def: matrix U}
        U = \boldsymbol{e}_1 \boldsymbol{s}^\top - \begin{bmatrix}
            G & \boldsymbol{0} \\
            \boldsymbol{0}^\top & 0
        \end{bmatrix} + \begin{bmatrix}
            0 & \boldsymbol{0}^\top \\
            \boldsymbol{0} & G
        \end{bmatrix}, \quad 
        \boldsymbol{e}_{1} = [1, 0, \ldots, 0]^\top \in \mathbb{R}^{k}.
\end{equation}
It is not difficult to find the following equations:
\begin{equation}
\label{eq:sumU}
    \sum_{i=1}^{k-m} U_{i,i+m} = s_{m}, \quad\text{for}\ m=0,1,\ldots,k-1.
\end{equation}
This condition is necessary and sufficient for a one-to-one correspondence between $G$ and $U$. This is summarized as the following lemma.

\begin{lemma} \label{lem:uniq-G}
Let $\boldsymbol{s} = [s_0,s_1,\ldots,s_{k-1}]^\top \in \mathbb{R}^k$. 
{Given an upper triangular matrix $U$ satisfying \eqref{eq:sumU}, the upper triangular matrix $G$ can be determined uniquely by \eqref{def: matrix U}.} Moreover, if $U$ is PSD, then {$G$ is also PSD}.
\end{lemma}
\begin{proof}  
Suppose that $U\in \mathbb{R}^{k\times k}$ is an upper triangular  matrix satisfying \eqref{eq:sumU}. Then, $U$ clearly takes the form:
\[
U = \boldsymbol{e}_1 \boldsymbol{s}^\top + 
\begin{bmatrix}
\tilde{u}_0 & \tilde{\boldsymbol{u}}^\top  \\
\boldsymbol{0} & \tilde{U}
\end{bmatrix},
\]
where
$\tilde{u}_0\in\mathbb{R}$ and $\tilde{\boldsymbol{u}}=[\tilde{u}_1,\ldots,\tilde{u}_{k-2},0]^\top \in \mathbb{R}^{k-1}$ with
\[
\tilde{u}_m = - \sum_{i=2}^{k-m} U_{i, i+m},\quad m = 0,\ldots,k-2,
\]
and $\tilde{U}\in \mathbb{R}^{(k-1)\times (k-1)}$ is the submatrix of $U$ with 
\[
\tilde{U}_{i,j} = U_{i+1,j+1},\quad 1\leq i, j \leq k-1.
\]
Then, $U$ admits the factorization \eqref{def: matrix U} if and only if $G$ satisfies
\begin{equation}
\label{eq:g=u}
- \begin{bmatrix}
G & \boldsymbol{0} \\
\boldsymbol{0}^\top & 0
\end{bmatrix}
+ \begin{bmatrix}
0 & \boldsymbol{0}^\top \\
\boldsymbol{0} & G
\end{bmatrix}
=
\begin{bmatrix}
\tilde{u}_0 & \tilde{\boldsymbol{u}}^\top  \\
\boldsymbol{0} & \tilde{U}
\end{bmatrix},
\end{equation}
which leads to
\begin{equation}
\label{eq:uniq-G}
-J^\top G J + G = \tilde{U}
\quad \text{with} \quad
J=\begin{bmatrix}
0 & 0 & \cdots & 0 & 0 \\
1 & 0 & \cdots & 0 & 0 \\
0 & 1 & \cdots & 0 & 0 \\
\vdots & \vdots & \ddots & \vdots & \vdots \\
0 & 0 & \cdots & 1 & 0
\end{bmatrix}
\in \mathbb{R}^{(k-1)\times (k-1)}.
\end{equation}

Define the linear operator  $\mathscr{A} X \coloneqq J^\top X J$ for any $X\in \mathbb{R}^{(k-1)\times (k-1)}$.
 Then, \eqref{eq:uniq-G} becomes 
 \begin{equation}
 \label{eq:g2tu}
 (\mathcal{I} - \mathscr{A}) G = \tilde{U},
 \end{equation}
where $\mathcal{I}$ is the identity operator. 
Since $J$ is nilpotent of index $k-1$, we have
\[
\mathscr{A}^m X = (J^\top)^m X J^m = \boldsymbol{0},\quad \text{for}\quad m\geq k-1,
\]
and hence
\[
    (\mathcal{I}-\mathscr{A})\braB{\sum_{m=0}^{k-2} \mathscr{A}^m}
    =\sum_{m=0}^{k-2} \mathscr{A}^m-\sum_{m=1}^{k-1} \mathscr{A}^m
    =\mathcal{I}-\mathscr{A}^{k-1}=\mathcal{I}.
\]
Thus, $\mathcal{I}-\mathscr{A}$ is invertible
\begin{equation}
\label{eq:inv_i-a}
(\mathcal{I}-\mathscr{A})^{-1} = \sum_{m=0}^{k-2} \mathscr{A}^m.
\end{equation} 
Combining \eqref{eq:g2tu} with \eqref{eq:inv_i-a} gives the unique upper triangular matrix:
\begin{equation}
\label{eq:recover-g}
G = \sum_{m=0}^{k-2} \mathscr{A}^m \tilde{U} = \sum_{m=0}^{k-2} (J^\top)^m \tilde{U} J^m.
\end{equation}
{It is straightforward to verify that $G$ in \eqref{eq:recover-g} also satisfies \eqref{eq:g=u}. Consequently,  for given $U$, there exists $G$ satisfying \eqref{eq:g=u} if and only if there exists $G$ satisfying  \eqref{eq:uniq-G}. Hence, \eqref{eq:g=u} admits a unique solution $G$ given by  \eqref{eq:recover-g}.}

Note that if $U$ is PSD, then its submatrix $\tilde{U}$ {is also PSD.} For any 
$\boldsymbol{x}\in\mathbb{R}^{k-1}$, we have
\[
\boldsymbol{x}^\top G \boldsymbol{x} = \sum_{m=0}^{k-2} \boldsymbol{x}^\top (J^\top)^m \tilde{U} J^m \boldsymbol{x}  = \sum_{m=0}^{k-2}  (J^m \boldsymbol{x})^\top \tilde{U} (J^m \boldsymbol{x}) \geq 0.
\]
{That is to say, $G$ is also PSD.}
\end{proof}

We now focus on the positive-definiteness condition \eqref{assumption: U}. It is clear that \eqref{assumption: U} holds if and only if the symmetric matrix $\frac{1}{2}(U+U^\top) - \gamma \boldsymbol{e}_{1} \boldsymbol{e}_{1}^\top$ is PSD. {Then, it admits a decomposition:}
\begin{equation}\label{eq:P-components}
    \frac{1}{2} (U + U^\top) - \gamma \boldsymbol{e}_{1} \boldsymbol{e}_{1}^\top  = \sum_{\ell = 1}^r \boldsymbol{p}_\ell \boldsymbol{p}_\ell^\top
\end{equation}
for some $1\leq r \leq k$, where $\boldsymbol{p}_\ell =  [p_1^{(\ell)}, p_{2}^{(\ell)}, \ldots, p_{k}^{(\ell)}]^\top$ for $1\leq \ell \leq r$ are not necessarily nonzero (it is possible that the right-hand side of \eqref{eq:P-components} has zero rank).
Once the vectors $\{\boldsymbol{p}_\ell\}_{\ell=1}^r$ are specified, the entries of $U$ are determined by
\begin{equation}\label{eq: U from vector p}
    U_{ij} = 
    \begin{cases}
        \displaystyle \sum_{\ell=1}^r (p^{(\ell)}_{i} )^{2} + \gamma \delta_{i,1}, & i = j, \\
        \displaystyle 2 \sum_{\ell = 1}^r p^{(\ell)}_{i} p^{(\ell)}_{j}, & i < j, \\
        0, & i > j.
    \end{cases}
\end{equation}
Substituting \eqref{eq: U from vector p} into \eqref{eq:sumU} leads to the following system of equations:
\begin{equation}\label{eq: vector p system}
\begin{cases}
    \displaystyle \sum_{\ell=1}^r \sum_{i=1}^{k} (p^{(\ell)}_{i})^{2} = s_{0} - \gamma, & m = 0, \\
    2\displaystyle \sum_{\ell=1}^r \sum_{i=1}^{k-m} p^{(\ell)}_{i} p^{(\ell)}_{i+m} = s_{m}, & m = 1, 2, \ldots, k-1.
\end{cases}
\end{equation}

Therefore, for a given vector $\boldsymbol{s}$, any solution of \eqref{eq: vector p system} for some $\gamma > 0$ and $1\leq r\leq k$ yields an explicit construction of an upper triangular matrix $U$ via \eqref{eq: U from vector p}. By construction, the resulting matrix $U$ satisfies the positive-definiteness condition \eqref{assumption: U} as well as \eqref{eq:sumU}. As shown in \cref{lem:uniq-G}, the upper triangular PSD  matrix $G$ is then uniquely determined by $U$. 

The system \eqref{eq: vector p system} is, in general, not easy to solve, particularly for the case $r \geq 2$.
However, we will show in the subsequent proposition that for a given vector $\boldsymbol{s}$ and $\gamma > 0$,
the solvability of system \eqref{eq: vector p system} for some $1 \leq r \leq k$ is equivalent to the solvability of the case of $r=1$.
For brevity, we omit the superscripts and write
$\boldsymbol{p} = [p_1,\ldots,p_k]^\top$, which leads to the following system:
\begin{equation}\label{eq: vector p system rank1}
\begin{cases}
     \displaystyle \sum_{i=1}^{k} p_{i}^{2} = s_{0} - \gamma, & m = 0, \\
    2\displaystyle \sum_{i=1}^{k-m} p_{i} p_{i+m} = s_{m}, & m = 1, 2, \ldots, k-1.
\end{cases}
\end{equation}

\begin{proposition}\label{thm:lr:equi}
Let $\boldsymbol{s} = [s_0,\ldots,s_{k-1}]^\top \in \mathbb{R}^k$ and $\gamma > 0$.
The following statements are equivalent:
\begin{enumerate}[label=(\roman*), font=\upshape]
\item There exists an upper triangular PSD matrix $G$ such that the positive-definiteness condition \eqref{assumption: U} holds.
\item The system \eqref{eq: vector p system} admits a real solution for some $1 \leq r \leq k$.
\item The system \eqref{eq: vector p system rank1} admits a real solution.
\item The inequality $T(x;\boldsymbol{s}) \geq \gamma$ holds for all $x \in [-1,1]$, where $T(x;\boldsymbol{s})$ is the real-valued generating polynomial \eqref{eq:def:chb} associated with $\boldsymbol{s}$.
\end{enumerate}
\end{proposition}

\begin{proof}
    We first establish the equivalence of $\text{(iii)}$ and $\text{(iv)}$, and then demonstrate the implications $\text{(i)} \Rightarrow \text{(ii)} \Rightarrow \text{(iii)} \Rightarrow \text{(i)}$.

\indent $\textrm{(iii)} \Leftrightarrow \textrm{(iv)}$: 
Assume that $\text{(iii)}$ holds. Let 
\begin{equation}
\label{eq:pz}
P(z) = p_1+p_2 z + \dotsb+p_k z^{k-1} \quad \text{with}\quad z = \mathrm{e}^{\mathrm{i}\theta}\quad \text{and}\quad \theta\in \mathbb{R}.
\end{equation}
A direct computation yields
\begin{equation}
\label{eq:pro1:temp1}
\begin{aligned}
|P(\mathrm{e}^{\mathrm{i}\theta})|^2 & = \sum_{j=1}^k \sum_{\ell=1}^k p_j p_\ell \mathrm{e}^{\mathrm{i} (j-\ell) \theta}
= \sum_{j=1}^k \sum_{\ell=1}^k p_j p_\ell \cos [(j-\ell)\theta] \\
& = \sum_{j=1}^{k} p_j^2 + 2 \sum_{m=1}^{k-1} \sum_{j=1}^{k-m} p_j p_{j+m} \cos m\theta.
\end{aligned}
\end{equation}
Substituting \eqref{eq: vector p system rank1} into \eqref{eq:pro1:temp1} gives
\begin{equation}
\label{eq:p2=m-g}
|P(\mathrm{e}^{\mathrm{i}\theta})|^2 = (s_0 - \gamma) + \sum_{m=1}^{k-1} s_m \cos(m\theta) = M(\theta;\boldsymbol{s}) - \gamma,
\end{equation}
where $M(\theta;\boldsymbol{s})$ is the real-valued generating function associated with $\boldsymbol{s}$ defined in \eqref{def: polynomial M(theta)}.
Since $|P(\mathrm{e}^{\mathrm{i}\theta})|^2 \geq 0$, it follows that
$M(\theta;\boldsymbol{s}) \geq \gamma$ for all $\theta\in[0,\pi]$ and then  $\text{(iv)}$ holds by \eqref{eq:same-min}. 

Conversely, assume that $\text{(iv)}$ holds. Then, by \eqref{eq:same-min}, $M(\theta;\boldsymbol{s})-\gamma$ is nonnegative. By the Fej\'er-Riesz theorem (see e.g., \cite[Lemma 6.1.3]{daubechies1992ten}), there exists a polynomial $P(z) = p_1+p_2 z + \dotsb+p_k z^{k-1}$, with real coefficients such that $M(\theta;\boldsymbol{s}) - \gamma = | P(\mathrm{e}^{\mathrm{i}\theta}) |^2$. Expanding this identity and comparing coefficients yields
the system \eqref{eq: vector p system rank1}, thereby implying $\text{(iii)}$.

\indent $\text{(i)} \Rightarrow \text{(ii)}$: Assume that $\text{(i)}$ holds. If the rank of $\frac{1}{2} (U + U^\top) - \gamma \boldsymbol{e}_{1} \boldsymbol{e}_{1}^\top$ is zero, then $\gamma=s_0$ and $s_m=0$ for $1\leq m \leq k-1$. In this case, the system \eqref{eq: vector p system} admits the trivial solution
by choosing $r=1$ and $\boldsymbol{p}_1=\boldsymbol{0}$.
Otherwise, we choose $r$ as the rank of $\frac{1}{2} (U + U^\top) - \gamma \boldsymbol{e}_{1} \boldsymbol{e}_{1}^\top$; then, $\text{(ii)}$ follows immediately from the derivation of the system \eqref{eq: vector p system}.

\indent $\text{(ii)} \Rightarrow \text{(iii)}$: Assume that $\text{(ii)}$ holds with $r>1$. We decompose $\boldsymbol{s}$ into $r$ components. For each $\ell = 1,\ldots,r$, define the coefficients:
\[
s^{(\ell)}_0 = \sum_{i=1}^k (p_i^{(\ell)})^2 + \frac{\gamma}{r},\quad 
s^{(\ell)}_m = 2\sum_{i=1}^{k-m} p_i^{(\ell)} p_{i+m}^{(\ell)}, \quad m = 1,\ldots,k-1.
\]
Based on the established equivalence $\text{(iii)} \Leftrightarrow \text{(iv)}$, the real-valued generating polynomial associated with the $\ell$-th component satisfies
\[
T(x; \boldsymbol{s}^{(\ell)}) = \sum_{m=0}^{k-1} s_m^{(\ell)} T_m(x) \geq \frac{\gamma}{r},\quad \forall x \in [-1,1],
\]
where $\boldsymbol{s}^{(\ell)} = [s_0^{(\ell)},\ldots, s_{k-1}^{(\ell)}]^\top$ and $T_m(x)$ is the $m$th Chebyshev polynomial of the first kind as aforementioned. Summing over $\ell$, and using $s_m = \sum_{\ell=1}^r s_m^{(\ell)}$, we obtain
\[
T(x;\boldsymbol{s}) = \sum_{m=0}^{k-1} s_m T_m(x) = \sum_{m=0}^{k-1} \sum_{\ell=1}^r s_m^{(\ell)}  T_m(x) 
= \sum_{\ell=1}^r T(x; \boldsymbol{s}^{(\ell)} ) \geq \gamma,\quad \forall x\in [-1,1].
\]
This implies $\text{(iv)}$, which is equivalent to $\text{(iii)}$.

\indent $\text{(iii)} \Rightarrow \text{(i)}$: Assume that $\text{(iii)}$ holds. We define $U$ according to \eqref{eq: U from vector p} with $r=1$. This construction ensures $ \frac{1}{2} (U + U^\top) - \gamma \boldsymbol{e}_{1} \boldsymbol{e}_{1}^\top = \boldsymbol{p}\boldsymbol{p}^\top$, which implies the positive-definiteness condition \eqref{assumption: U}, as well as \eqref{eq:sumU}. Since $U$ is PSD, the existence of the upper triangular PSD matrix $G$ follows from \cref{lem:uniq-G}.
\end{proof}
\begin{remark}
The equivalence in \cref{thm:lr:equi} remains valid in the limiting case of $\gamma=0$, from the same proof. In this case, the decomposition \eqref{eq:structure2} reduces to \eqref{eq:structure-general} under the simple multiplier, and the corresponding condition becomes the non-negativity of the associated generating polynomial on $[-1,1]$. In the energy-dissipation analysis of IMEX-LMMs for gradient flows, we use the case of $\gamma>0$, which provides a positive quadratic term used to control the explicit treatment of the nonlinear term.
\end{remark}

\subsection{Determination of \texorpdfstring{$\gamma$}{gamma} and \texorpdfstring{$G$}{G}} \label{subsec:construct-g&G}
Let $\gamma_{\max}$ be the maximum value such that there exists an upper triangular PSD matrix $G$ satisfying the positive-definiteness condition \eqref{assumption: U}. Then, by the equivalence of $\text{(i)}$ and $\text{(iv)}$ in \cref{thm:lr:equi}, for a given vector $\boldsymbol{s} = [s_0,\ldots,s_{k-1}]^\top$, $\gamma_{\max}$ is given by
\begin{equation}
\label{eq:max-gamma}
\gamma_{\max} = \min_{x \in [-1,1]} T(x; \boldsymbol{s}).
\end{equation}
Consequently, if $\gamma_{\max}> 0$, by selecting $\gamma$ such that $0 < \gamma \leq \gamma_{\max}$, and the equivalence of $\text{(iii)}$ and $\text{(iv)}$ in \cref{thm:lr:equi}, there exists a vector $\boldsymbol{p}$ satisfying \eqref{eq: vector p system rank1}, from which the matrix $U$ can be constructed by \eqref{eq:P-components} with $r=1$ and {$G$ then can be constructed} by \eqref{eq:recover-g}.

We now focus on how to solve \eqref{eq: vector p system rank1}. 
We define the Laurent polynomial:
\begin{equation}\label{def:L(z)}
    L(z) \coloneqq s_0-\gamma + \sum_{m=1}^{k-1} \frac{s_m}{2} (z^m + z^{-m}),\quad z\in \mathbb{C}.
\end{equation}
Note that $L(\mathrm{e}^{\mathrm{i}\theta}) = M(\theta;\boldsymbol{s})-\gamma$ for $\theta\in\mathbb{R}$, where $M(\theta;\boldsymbol{s})$ is the real-valued function associated with $\boldsymbol{s}$ defined in \eqref{def: polynomial M(theta)}.
With the definition \eqref{eq:pz} of $P(z)$ and by \eqref{eq:p2=m-g}, we have the factorization $L(z) = P(z) P(z^{-1})$ for all $z$ satisfying $|z|=1$, which implies $$L(z) = P(z) P(z^{-1}),\quad \forall z \in \mathbb{C} \setminus \{0\}.$$
We next determine the polynomial $P(z)$, whose coefficients form the vector $\boldsymbol{p}$ satisfying \eqref{eq: vector p system rank1}, via the roots of $L(z)$.

If $P(z) \equiv 0$, we have $\boldsymbol{p}=\boldsymbol{0}$ and $L(z) \equiv 0$. Otherwise, by the fundamental theorem of algebra, $P(z)$ can be factorized into
\begin{equation}\label{factorization:P(z)}
    P(z) = a_{0} \prod_{\{j: \mathrm{Im}{(z_j)}>0\}} (z-z_j)^{m_j} (z-z_j^*)^{m_j} \prod_{\{l: z_l \in \mathbb{R}\}} (z-z_l)^{m_l},
\end{equation}
where $(z_j,z_j^*)$ denote the $j$th root pair with multiplicity $m_j$ and $a_0\neq 0$ is some real constant. Consequently, $L(z)$ can be factorized into
\begin{equation}\label{factorization:L(z)}
\begin{aligned}
L(z) &= a_0^2 
\prod_{\{j: \mathrm{Im}(z_j)>0\}} 
(z-z_j)^{m_j} (z-z_j^*)^{m_j} (z^{-1}-z_j)^{m_j} (z^{-1}-z_j^*)^{m_j} \\
&\quad \cdot \prod_{\{l: z_l \in \mathbb{R}\}} 
(z-z_l)^{m_l} (z^{-1}-z_l)^{m_l}.
\end{aligned}
\end{equation}
Therefore, once the roots of $L(z)$ defined in \eqref{def:L(z)} are computed, we then obtain the factorization \eqref{factorization:L(z)} of $L(z)$ and consequently the factorization \eqref{factorization:P(z)} of $P(z)$, which gives the real coefficients $p_i, ~i = 1,\ldots,k$. Here, we emphasize that the choice of $P(z)$ is not unique and one can simply choose the roots $z_j$ satisfying $|z_j|\leq 1$.

Based on the above analysis, the following procedure summarizes the construction of
$\gamma$ and the associated matrix $G$ for a given vector $\boldsymbol{s}$,
such that the positive-definiteness condition \eqref{assumption: U} is satisfied.

\begin{enumerate}[label=\textit{Step \arabic*.}, leftmargin=4em, itemsep=1ex]
\item Compute $\gamma_{\max}$ according to \eqref{eq:max-gamma}. If $\gamma_{\max} \leq 0$, then no matrix $G$ satisfies the positive-definiteness condition \eqref{assumption: U}. Otherwise, select a $\gamma$ such that $0 < \gamma \leq \gamma_{\max}$.
\item Consider $L(z)$ defined in \eqref{def:L(z)}. If $L(z)=0$, then $G=\boldsymbol{0}$ is obtained. Otherwise, by computing the roots of the Laurent polynomial $L(z)$, derive the factorization \eqref{factorization:L(z)} of $L(z)$ and then the factorization \eqref{factorization:P(z)} of $P(z)$. The vector $\boldsymbol{p}$ is then computed from the coefficients of $P(z)$.
\item Compute $U$ via \eqref{eq:P-components} with $r=1$, and then obtain $G$ from \eqref{eq:recover-g}.
\end{enumerate}
\section{Construction of energy-dissipative IMEX-LMMs}\label{sec:ConstructionLMM}

This section is devoted to the construction of energy-dissipative IMEX-LMMs within the framework of \cref{thm:main-result}. We begin by revisiting the well-known IMEX-BDF$k$ methods, and then present a general approach for constructing high-order IMEX-LMMs, which gives a sixth-order energy-dissipative IMEX-LMM for the first time (to the best of our knowledge). Finally, we provide a rigorous proof of the nonexistence of seventh-order energy-dissipative IMEX-LMMs for the multiplier $v^{n+1} = u^{n+1}-u^n$ in \eqref{eq:structure2}.

\subsection{A revisit of IMEX-BDF methods}

By setting $B_i^{(k)}=\delta_{i,0}$ in \eqref{eq:model:lmm2} and determining the coefficients
$\{A_i^{(k)}\}_{i=0}^k$ and $\{\hat{B}_i^{(k)}\}_{i=1}^k$
from the order conditions in \eqref{constraint: order conditions of LMM} and \eqref{condition: normalization},
we obtain the IMEX-BDF$k$ methods. 
Consider the reformulation \eqref{scheme: LMM-difference form} of IMEX-BDF$k$ methods:
\begin{equation}\label{scheme: BDFk-difference form}
\begin{aligned}
&\frac{1}{\tau} \sum_{i=0}^{k-1} a_{i}^{(k)} \mathcal{M}^{-1}\brab{\delta u^{n+1-i}}
- \frac{1}{2} \mathcal{L} (u^{n+1} + u^{n}) - \frac{1}{2} \mathcal{L} \delta u^{n+1} \\
&\qquad\qquad\qquad\qquad= f(u^n) + \sum_{i=1}^{k-1} \hat{b}_i^{(k)} \delta f(u^{n+1-i}) \quad\text{for $k-1\leq n \leq N-1$,}
\end{aligned}
\end{equation}
where the coefficients $a_{i}^{(k)}$ and $\hat{b}_i^{(k)}$ are listed in \cref{tab:BDF-coefficients2}.

\begin{table}[htbp]
\centering
\begin{tabular}{c|cccccc}
    \toprule
    $k$ & $a_0^{(k)}$ & $a_1^{(k)}$ & $a_2^{(k)}$ & $a_3^{(k)}$ & $a_4^{(k)}$ & $a_5^{(k)}$ \\
    \midrule
    1 & $1$ \\[4pt]
    2 & $\frac{3}{2}$ & $-\frac{1}{2}$ \\[4pt]
    3 & $\frac{11}{6}$ & $-\frac{7}{6}$ & $\frac{1}{3}$ \\[4pt]
    4 & $\frac{25}{12}$ & $-\frac{23}{12}$ & $\frac{13}{12}$ & $-\frac{1}{4}$ \\[4pt]
    5 & $\frac{137}{60}$ & $-\frac{163}{60}$ & $\frac{137}{60}$ & $-\frac{21}{20}$ & $\frac{1}{5}$ \\[4pt]
    6 & $\frac{49}{20}$ & $-\frac{71}{20}$ & $\frac{79}{20}$ & $-\frac{163}{60}$ & $\frac{31}{30}$ & $-\frac{1}{6}$ \\
    \bottomrule
\end{tabular}
\quad 
\begin{tabular}{ccccc}
    \toprule
     $\hat{b}_1^{(k)}$ & $\hat{b}_2^{(k)}$ & $\hat{b}_3^{(k)}$ & $\hat{b}_4^{(k)}$ & $\hat{b}_5^{(k)}$   \\
    \midrule
      \\[4pt]
     $1$  \\[4pt]
     $2$ & $-1$ \\[4pt]
     $3$ & $-3$ & $1$  \\[4pt]
     $4$ & $-6$ & $4$ & $-1$  \\[4pt]
     $5$ & $-10$ & $10$ & $-5$ & $1$  \\
    \bottomrule
\end{tabular}
\caption{Coefficients $a_i^{(k)}$ and $\hat{b}_i^{(k)}$ for IMEX-BDF$k$ methods.}
\label{tab:BDF-coefficients2}
\end{table}

\begin{lemma} \label{lem:bdf2-6}
The conditions in \eqref{eq:lr:es2} hold for IMEX-BDF$k$ methods \eqref{scheme: BDFk-difference form} with $1 \leq k \leq 5$, but fail for the IMEX-BDF6 method.
\end{lemma}
\begin{proof}
Since $\boldsymbol{b}^{(k)} = [1/2, 0, \dots, 0]^\top \in \mathbb{R}^{k}$ is fixed, the associated real-valued generating polynomial $T(x; \boldsymbol{b}^{(k)}) \equiv 1/2$ trivially satisfies the second inequality in \eqref{eq:lr:es2}. 
Therefore, it suffices to verify the positivity of $T(x;\boldsymbol{a}^{(k)})$.
In the case of IMEX-BDF$1$ method, $T(x; \boldsymbol{a}^{(1)}) \equiv 1 > 0$ naturally holds. 
For $2 \le k \le 5$, the corresponding generating polynomials
$T(x;\boldsymbol{a}^{(k)})$, together with their minimum values
$\min_{x\in[-1,1]} T(x;\boldsymbol{a}^{(k)})$, are listed in
\cref{tab:BDF-poly}.
For intuition, the graphs of $T(x;\boldsymbol{a}^{(k)})$
for $2 \le k \le 5$ are shown in \cref{fig:bdf2-5}.
For brevity, we omit the details of the computation of
$\min_{x\in[-1,1]} T(x;\boldsymbol{s})$.
For IMEX-BDF6 method, we evaluate the real-valued generating polynomial $T(x; \boldsymbol{a}^{(6)})$ at a specific point $x_0 = 0$, then
\[
\begin{aligned}
T(x_0; \boldsymbol{a}^{(6)}) & = \frac{49}{20} T_0(x_0) -\frac{71}{20} T_1(x_0) + \frac{79}{20} T_2(x_0) - \frac{163}{60} T_3(x_0) + \frac{31}{30} T_4(x_0) - \frac{1}{6} T_5(x_0)\\
& = \frac{49}{20} -  \frac{79}{20} + \frac{31}{30} = -\frac{7}{15}<0,
\end{aligned}
\]
which violates \eqref{eq:lr:es2}.
\end{proof}

\begin{table}[htbp]
\centering
\begin{tabular}{c|ccc}
    \toprule
    $k$ & Real-valued generating polynomial $T(x; \boldsymbol{a}^{(k)})$ & Minimum value \\
    \midrule
    2 & $\frac{3}{2}  - \frac{1}{2} x$ & $1$ \\[4pt]
    3 & $\frac{3}{2} -\frac{7}{6}x + \frac{2}{3} x^2$ & $\frac{95}{96}  \approx 0.989583$ \\[4pt]
    4 & $1 -\frac{7}{6}x + \frac{13}{6} x^2-x^3$ &  $\frac{664}{729} - \frac{43\sqrt{43}}{2916}\approx 0.814139$ \\[4pt]
    5 & $\frac{1}{5} +\frac{13}{30}x + \frac{89}{30} x^2-\frac{21}{5}x^3 + \frac{8}{5} x^4$ & $ \approx  0.185546$ \\
    \bottomrule
\end{tabular}
\caption{Real-valued generating polynomials for IMEX-BDF$k$ ($2\leq k \leq 5$) methods.}
\label{tab:BDF-poly}
\end{table}


\begin{figure}[t]
    \centering
    \begin{subfigure}[b]{0.45\textwidth}
        \centering
        \includegraphics[width=\linewidth]{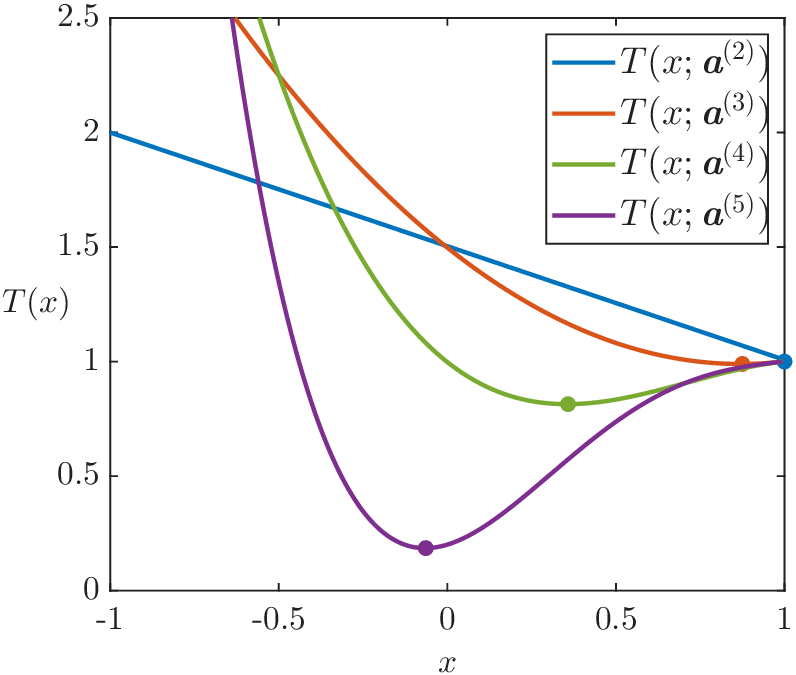}
        \caption{IMEX-BDF methods of orders 2–5.}
        \label{fig:bdf2-5}
    \end{subfigure}%
    \quad
    \begin{subfigure}[b]{0.45\textwidth}
        \centering
        \includegraphics[width=\linewidth]{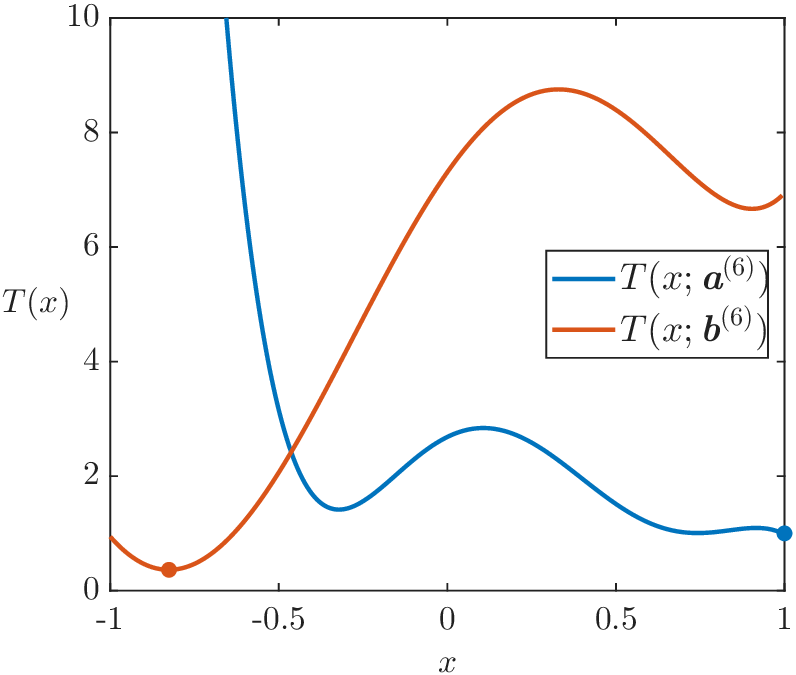}
        \caption{IMEX-LMM of order 6.}
        \label{fig:lmm6}
    \end{subfigure}
    \caption{The real-valued generating polynomials $T(x;\boldsymbol{a}^{(k)})$ associated with the IMEX-BDF methods of orders $k=2,3,4,$ and $5$ are shown on the left, where the corresponding coefficient vectors $\boldsymbol{a}^{(k)}$ are listed in \cref{tab:BDF-coefficients2}. The polynomials $T(x;\boldsymbol{a}^{(6)})$ and $T(x;\boldsymbol{b}^{(6)})$ associated with the IMEX-LMM of order $6$ are displayed on the right, with coefficients given in \cref{tab:lmm6}. For each method, the minimum is marked by a solid dot.}
    \label{fig:gen-polys}
\end{figure}

We now show the computation of the parameters $\alpha$, $\beta$, $G_a$, and $G_b$ in \cref{thm:main-result} for IMEX-BDF$k$ methods with $2\leq k \leq 5$ by following the procedure in \cref{subsec:construct-g&G}.  
Let $\alpha_{\max}$ and $\beta_{\max}$ denote the maximum values of $\alpha$ and $\beta$ satisfying conditions in \eqref{eq:lr:es2}. 
The coefficient vector $\boldsymbol{b}^{(k)} = [1/2,0,\ldots,0]^\top$ 
implies that $T(x; \boldsymbol{b}^{(k)}) \equiv 1/2$, immediately yielding $\beta_{\max} = 1/2$. Consequently, the associated matrix  $G_b=\boldsymbol{0}\in\mathbb{R}^{(k-1)\times(k-1)}$. The value of
$\alpha_{\max}$ is the minimum of $T(x;\boldsymbol{a}^{(k)})$ on $[-1,1]$, which is obtained from  
{\cref{tab:BDF-poly}.}
We provide closed-form expressions for $G_a$ in the cases $k=2$ and $3$, and use numerical approximations in the cases  $k=4$ and $5$. The calculated $\alpha_{\max}$ and the associated matrix $G_a$ are summarized in \cref{tab:bdf-g}.

\begin{table}[htbp]
\centering
\begin{tabular}{c|cc}
\toprule
$k$ &  $G_a$ & $\alpha_{\max}$ \\[4pt] 
\midrule

$2$ & $\frac{1}{4}$ & $1$ \\[4pt]
$3$ &
$\begin{bmatrix}
\frac{65}{96} & -\frac{7}{12} \\
0 & \frac{1}{6}
\end{bmatrix}$ 
& $\frac{95}{96}$ \\[4pt]
$4$ & 
$\begin{bmatrix}
1.219233 & -1.595139 & 0.804452 \\
 & 0.701926 & -0.697752 \\
 & & 0.312746
\end{bmatrix}$
& 0.814139 \\[4pt]
$5$ & 
$\begin{bmatrix}
2.084535 & -2.591196 & 2.073333 & -0.946751 \\
& 1.787561 & -1.597106 & 1.584577 \\
& & 0.955657 & -0.779076\\
& & & 0.754560
\end{bmatrix}$ & 0.185546\\
\bottomrule
\end{tabular}
\caption{$\alpha_{\max}$ and the associated  $G_a$ for IMEX-BDF$k$ methods.}
\label{tab:bdf-g}
\end{table}

\begin{remark}
\cref{lem:bdf2-6} establishes the energy-dissipative property of IMEX-BDF$k$ methods with $1\leq k \leq 5$. Comparing to the energy analysis of IMEX-BDF methods in \cite{li2021f} that solves the quadratic system \eqref{eq: vector p system rank1} case-by-case for different values of $k$, this work provides a unified framework based on checking the positivity of real-valued generating polynomials and moreover, proves rigorously the equivalence theory in Proposition \ref{thm:lr:equi}. In addition, \cite[Appendix A]{LiaoKang-2024IMA-PFC} proposed an alternative strategy to construct the PSD matrix $G_a$ for BDF3 method with optimal $\alpha = \tfrac{95}{96} = \alpha_{\max}$, while it gives $\alpha = \tfrac{4919}{6144}\approx0.801 < \alpha_{\max}$ for BDF4 method and $\alpha = \tfrac{646631}{3840000} \approx 0.168 < \alpha_{\max}$ for BDF5 method, with $\alpha_{\max}$ given in Table \ref{tab:bdf-g}.
\end{remark}

\subsection{General approach to construct energy-dissipative IMEX-LMMs}

\Cref{lem:bdf2-6} shows that the IMEX-BDF6 method does not satisfy the positivity conditions under the simple multiplier used in \cref{thm:main-result}.
This failure motivates the construction of general energy-dissipative IMEX-LMMs with high-order accuracy. 

To construct a $k$th-order energy-dissipative IMEX-LMM, the order conditions in \eqref{constraint: order conditions of LMM}--\eqref{condition: normalization} are required first and then the conditions in \eqref{eq:lr:es2} shall be imposed for energy dissipation according to the framework of \cref{thm:main-result}. 

To represent the coefficients of IMEX-LMMs subject to the order conditions in \eqref{constraint: order conditions of LMM} and \eqref{condition: normalization}, we introduce a set of parameters $w_0,\ldots,w_{k} \in \mathbb{R}$ so that the coefficients $\{A_i^{(k)}\}_{i=0}^k$, $\{B_i^{(k)}\}_{i=0}^k$, and $\{\hat{B}_i^{(k)}\}_{i=1}^k$ in \eqref{constraint: order conditions of LMM}--\eqref{condition: normalization}  satisfy
\[
\sum_{i=0}^k A_i^{(k)}(-i)^{m+1}  =  (m+1) \sum_{i=0}^k B_i^{(k)}(-i)^m = (m+1) \sum_{i=1}^k \hat{B}_i^{(k)}(-i)^m = w_m
\]
for $m=0,\ldots,k-1$ and $\sum_{i=0}^k A_i^{(k)} = 0$. Note that $w_0 = 1$ due to the normalization condition \eqref{condition: normalization} and $w_1,\ldots,w_{k-1}$  are free parameters. Further, taking $w_k = B_k^{(k)}$ as an additional free parameter, the above equations can be written in matrix form of
\begin{equation}
\label{eq:oc-matform}
 W_1
 \begin{bmatrix}
 A_0^{(k)}\\
 \vdots \\
 A_k^{(k)}
 \end{bmatrix}
 =
 \begin{bmatrix}
 0 \\
 \tilde{\boldsymbol{w}}
 \end{bmatrix},
 \quad 
 W_2 
 \begin{bmatrix}
  B_0^{(k)}\\
 \vdots \\
 B_{k-1}^{(k)}
 \end{bmatrix}
 = D_k \tilde{\boldsymbol{w}} - w_k \boldsymbol{z},
 \quad 
 W_3 
  \begin{bmatrix}
  \hat{B}_1^{(k)}\\
 \vdots \\
 \hat{B}_{k}^{(k)}
 \end{bmatrix}
 = D_k \tilde{\boldsymbol{w}},
 \end{equation}
where
\[
\begin{aligned}
\boldsymbol{w} & = [w_1,\ldots,w_{k-1},w_k]^\top,\quad\tilde{\boldsymbol{w}} = [w_0,w_1,\ldots,w_{k-1}]^\top, \\
D_k & = \operatorname{diag}\braB{1,\frac{1}{2},\ldots,\frac{1}{k}}, \quad  \boldsymbol{z} = \kbrab{1, -k,\ldots,(-k)^{k-1}}^\top,
\end{aligned}
\]
and $W_1, W_2, W_3$ are the transposed Vandermonde matrices:
\[
W_1 = W(0,-1,\ldots,-k)^\top, \quad W_2 = W(0,-1,\ldots,-k+1)^\top,\quad W_3 = W(-1,\ldots,-k)^\top.
\]
Here,
$W(x_0, \dots, x_n) \in \mathbb{R}^{(n+1)\times(n+1)}$ is the Vandermonde matrix generated by $x_0, \dots, x_n$, i.e., $W_{i+1,j+1} = x_i^j$ for $0\le i,j \le n$. 
Substituting \eqref{eq:oc-matform} into \eqref{def: coefficients a b c} yields the explicit formulas of the coefficient vectors:
\begin{equation}
\label{eq:hom:ab}
\begin{gathered}
\boldsymbol{a}^{(k)} = [E_k, \boldsymbol{0}] W_1^{-1}
 \begin{bmatrix}
0 \\
\tilde{\boldsymbol{w}}
\end{bmatrix},
\\
\boldsymbol{b}^{(k)} = E_k W_2^{-1} (D_k \tilde{\boldsymbol{w}} - w_k \boldsymbol{z}) - 
\begin{bmatrix}
{1}/{2} \\
1 \\
\vdots\\
1
\end{bmatrix},
\quad
\hat{\boldsymbol{b}}^{(k)} = 
E_k W_3^{-1} D_k \tilde{\boldsymbol{w}} - 
\begin{bmatrix}
1 \\
1 \\
\vdots\\
1
\end{bmatrix},
\end{gathered}
\end{equation}
where $\boldsymbol{a}^{(k)}$, $\boldsymbol{b}^{(k)}$, and  $\hat{\boldsymbol{b}}^{(k)}$ are defined in \eqref{eq:coef-vecs}, and $E_k\in \mathbb{R}^{k\times k}$ denotes the lower triangular matrix of ones, i.e., $(E_k)_{ij} = 1$ if $i \ge j$, and $(E_k)_{ij} =0$ otherwise.

Then, the construction of a $k$th-order energy-dissipative IMEX-LMM reduces to the following feasibility problem:

\begin{enumerate}[label=\textnormal{(FP1)}, ref=FP1]
    \item \label{FP1} 
    Find $w_1, \dots, w_{k-1}, w_k \in \mathbb{R}$ such that:
    \begin{equation}
    \label{eq:hom:obj}
    \min_{x\in[-1,1]} T(x;\boldsymbol{a}^{(k)}) >0\quad \text{and}\ \min_{x\in[-1,1]} T(x; \boldsymbol{b}^{(k)}) > 0,
    \end{equation}
where $\boldsymbol{a}^{(k)}$ and $\boldsymbol{b}^{(k)}$ are determined by \eqref{eq:hom:ab}.
\end{enumerate}
Notably, $\hat{\boldsymbol{b}}^{(k)}$ is uniquely determined by $\tilde{\boldsymbol{w}}$ and does not enter into the problem \eqref{FP1}.

To solve \eqref{FP1}, one must accurately compute the minima of $T(x;\boldsymbol{a}^{(k)})$ and $T(x;\boldsymbol{b}^{(k)})$ on $[-1,1]$. In the remainder of this subsection, we present an efficient numerical method for this task based on the properties of Chebyshev polynomials. For brevity, let $T(x) = T(x;\boldsymbol{s})$, where $\boldsymbol{s} = [s_0, \dots, s_{k-1}]^\top$ corresponds to either $\boldsymbol{a}^{(k)}$ or $\boldsymbol{b}^{(k)}$. 
By \eqref{eq:coef-vecs}, $T(x)$ has the Chebyshev expansion $T(x) = \sum_{m=0}^{k-1} s_m T_m(x)$, and its derivative can be expressed in the same basis:
\[
T^\prime (x) = \sum_{m=0}^{k-1} s_m T_m^\prime (x) = \sum_{m=0}^{k-1} \hat{s}_m T_m (x),
\]
where the coefficients $\hat{s}_m$ are determined via the backward recurrence relations \cite[Eq. (3.234)]{shen2011spectral}:
\[
\left\{
\begin{aligned}
&\hat{s}_{k-1} = 0,\ \hat{s}_{k-2} = 2 (k-1) {s}_{k-1}, \\
& \hat{s}_{m-1} = (2 m {s}_m + \hat{s}_{m+1}) / (1+\delta_{m,1}),\quad m=k-2,\ldots,1.
\end{aligned}
\right.
\]
The roots of $T^\prime(x)$ are equivalent to the eigenvalues of the colleague matrix $C_k$ constructed from the coefficients $\hat{s}_m$. This matrix is tridiagonal except for its last row and takes the form \cite[Theorem 18.1]{trefethen2019approximation}:
\[
C_k = \begin{bmatrix}
0 & 1 & & & & \\
\frac{1}{2} & 0 & \frac{1}{2} & & & \\
& \frac{1}{2} & 0 & \frac{1}{2} & & \\
& & \ddots & \ddots & \ddots & \\
& & & & & \frac{1}{2} \\
& & & & \frac{1}{2} & 0
\end{bmatrix}
-\frac{1}{2 \hat{s}_{k-2}}
\begin{bmatrix} 
& & & & \\
& & & & \\
& & & & \\
& & & & \\
\\
\hat{s}_0 & \hat{s}_1 &  & \ldots & \hat{s}_{k-3}
\end{bmatrix}.
 \]
 This eigenvalue-based approach
 provides a robust framework for numerical root finding, as implemented in the Chebfun package \cite{driscoll2014chebfun}. 

Consequently, the eigenvalues of $C_k$ provide a set of candidate local extrema, denoted by $\{x_j^*\}_{j=1}^{k-2}$.  After discarding repeated roots and those outside $[-1,1]$, we augment this set with the endpoints $x_0^* = -1$ and $x_{k-1}^* = 1$. The global minimum is then given by
 \begin{equation}
 \label{eq:disc-min}
 \min_{x\in[-1,1]} T(x) = \min_{0\leq j\leq k-1} T(x_j^*(\boldsymbol{s})),
 \end{equation}
where for notational convenience, we retain the index range $0\leq j\leq k-1$ and $x_j^*(\boldsymbol{s})$ emphasizes the nonlinear dependence of the critical points on $\boldsymbol{s}$.

\subsection{A sixth-order energy-dissipative IMEX-LMM}

Solving the feasibility problem \eqref{FP1} for $k=6$ yields a family of sixth-order energy-dissipative IMEX-LMMs. We provide a representative instance defined by the parameters:
\[
w_0 = 1, \ w_1 = \frac{64}{5},\ w_2 = -\frac{141}{5}, \ w_3 = 111, \ w_4 = -1034, \ w_5 = 9886, \  w_6 =  -\frac{23}{100}.
\]
The resulting sixth-order IMEX-LMM is specified in \cref{tab:lmm6}.

\begin{table}[htbp]
    \centering
    \setlength{\extrarowheight}{9pt}
    \begin{tabular}{c|ccccccc}
    \toprule
        $i$ & $0$ & $1$ & $2$ & $3$ & $4$ & $5$ & $6$ \\
    \midrule
       $A_i^{(k)}$  & $\dfrac{2617}{200}$ & $-\dfrac{6897}{200}$  & $\dfrac{4481}{120}$ & $-\dfrac{319}{12}$ & $\dfrac{647}{40}$ & $-\dfrac{4231}{600}$ & $\dfrac{911}{600}$ \\
       
       $B_i^{(k)}$ & $\dfrac{1525}{288}$ & $-\dfrac{2999}{7200}$ & $-\dfrac{4001}{720}$ & $\dfrac{79}{144}$ & $\dfrac{557}{288}$ &  $-\dfrac{827}{1440}$  &  $-\dfrac{23}{100}$ \\ 

       $\hat{B}_i^{(k)}$ &  & $\dfrac{225751}{7200}$ & $-\dfrac{122377}{1440}$ & $\dfrac{15329}{144}$ & $-\dfrac{11159}{144}$ & $\dfrac{44923}{1440}$ & $-\dfrac{39781}{7200}$ \\
       
       $a_i^{(k)}$ & $\dfrac{2617}{200}$  &  $-\dfrac{107}{5}$  &  $\dfrac{1913}{120}$ & $-\dfrac{1277}{120}$ &  $\dfrac{83}{15}$   &  $-\dfrac{911}{600}$  & \\

       $b_i^{(k)}$ & $\dfrac{1381}{288}$ & $\dfrac{13963}{3600}$ & $-\dfrac{1007}{600}$ & $-\dfrac{4067}{3600}$  & $\dfrac{5791}{7200}$  &  $\dfrac{23}{100}$ &  \\

       $\hat{b}_i^{(k)}$ &  &  $\dfrac{218551}{7200}$ & $-\dfrac{196667}{3600}$ & $\dfrac{31093}{600}$ & $-\dfrac{92417}{3600}$ & $\dfrac{39781}{7200}$ & \\[5pt]
    \bottomrule
    \end{tabular}
    \caption{A sixth-order energy-dissipative IMEX-LMM.}
    \label{tab:lmm6}
\end{table}

A direct verification confirms that the scheme listed in \cref{tab:lmm6} satisfies the order conditions in \eqref{constraint: order conditions of LMM} and  \eqref{condition: normalization}.
Furthermore, using \eqref{def: coefficients a b c}, we derive the coefficient vectors $\boldsymbol{a}^{(k)}$, $\boldsymbol{b}^{(k)}$, and $\hat{\boldsymbol{b}}^{(k)}$ as in \eqref{eq:coef-vecs}, whose components are also provided in \cref{tab:lmm6}, and verify that the constraints in \eqref{eq:hom:obj} hold.  
This verification is further illustrated by the plots of $T(x;\boldsymbol{a}^{(k)})$ and $T(x;\boldsymbol{b}^{(k)})$ shown in \cref{fig:lmm6}.

We now address the computation of the parameters $\alpha$, $\beta$, $G_a$, and $G_b$ in \cref{thm:main-result} for  this sixth-order IMEX-LMM by following the procedure in \cref{subsec:construct-g&G}. Let $\alpha_{\max}$ and $\beta_{\max}$ denote the maximum values of $\alpha$ and $\beta$ allowed by the conditions in \eqref{eq:lr:es2}. Then, 
\[
\alpha_{\max} = \min_{x\in [-1,1]} T(x;\boldsymbol{a}^{(k)}) = 1,\quad 
\beta_{\max} = \min_{x\in[-1,1]} T(x;\boldsymbol{b}^{(k)}) \approx 0.363757.
\]
The associated $G_a$ and $G_b$ are 
\[
G_a = 
\begin{bmatrix}
    11.525734 & -19.376783 & 13.720328  & -8.056395 & 2.746382 \\
    & 9.695922 & -15.358795 & 9.044053 & -3.015320 \\
    & & 7.490199 & -10.224600 & 3.509334 \\
    & & & 4.502521 & -3.783101 \\
    & & & & 1.030518
\end{bmatrix},
\]
and
\[
G_b = 
\begin{bmatrix}
    4.424381 & 3.844372 & -1.572744 & -1.143033 & 0.562442 \\
    & 4.382517 & 4.102580 & -1.605295 & -1.734487 \\
    & & 3.984379 & 4.202963 & 0.218659 \\
    & & & 3.978051 & 3.973025 \\
    & & & & 1.889072
\end{bmatrix}.
\]
The minimum eigenvalue of $G_a+G_a^\top$ is approximately $0.078211$, while the minimum eigenvalue of $G_b+G_b^\top$ is approximately $0.406943$.

We next examine the classical linear stability of the sixth-order IMEX-LMM listed in \cref{tab:lmm6} and compare it with the IMEX-BDF6 method. 
For a
given IMEX-LMM, let
\[
\rho(\xi)=\sum_{i=0}^{6}A_i^{(6)}\xi^{6-i},\qquad
\sigma(\xi)=\sum_{i=0}^{6}B_i^{(6)}\xi^{6-i},\qquad
\hat\sigma(\xi)=\sum_{i=1}^{6}\hat B_i^{(6)}\xi^{6-i}.
\]
Here $\rho$ is the first characteristic polynomial associated with the
difference operator. For the new sixth-order IMEX-LMM in \cref{tab:lmm6}, a
direct computation shows that $\rho$ satisfies the root condition \cite{Dahlquist1956-0stab}, namely,
all its roots lie in the closed unit disk and all roots on the unit circle are
simple. Hence the method is zero-stable.

To illustrate the IMEX stability behavior, we apply the method to the split scalar test equation as in \cite{Frank-1997ANM-stabEq}
\[
    y'=\lambda_I y+\lambda_E y,\quad
    z_I=\tau\lambda_I,\quad z_E=\tau\lambda_E,
\]
where the $\lambda_I$-term is treated implicitly and the $\lambda_E$-term
explicitly. This gives the characteristic equation
\begin{equation}
\label{eq:char-eq}
    \rho(\xi)-z_I\sigma(\xi)-z_E\hat{\sigma}(\xi)=0.
\end{equation}
The IMEX stability region is defined as the set of all
$(z_I,z_E)\in\mathbb C^2$ for which \eqref{eq:char-eq} satisfies the root
condition. Since this region is four-dimensional over $\mathbb R$, we display
representative two-dimensional slices.

For comparison, we first plot the corresponding slices for IMEX-BDF6. In this
case,
\[
\begin{aligned}
\rho(\xi)
&= \frac{49}{20}\xi^6 - 6\xi^5 + \frac{15}{2}\xi^4
   - \frac{20}{3}\xi^3 + \frac{15}{4}\xi^2
   - \frac{6}{5}\xi + \frac{1}{6},\\
\sigma(\xi)&=\xi^6,\\
\hat{\sigma}(\xi)
&=6\xi^5-15\xi^4+20\xi^3-15\xi^2+6\xi-1.
\end{aligned}
\]
The purely implicit slice $z_E=0$, the purely explicit slice $z_I=0$, and
the $z_E$-plane slices for fixed $z_I=0,-1,-5,-10$ are shown in \cref{subfig:bdf6-implicit,subfig:bdf6-explicit,subfig:bdf6-imex}. In
\cref{subfig:bdf6-implicit,subfig:bdf6-explicit}, the stability regions are shaded in light
red. In \cref{subfig:bdf6-imex}, the stability region for each fixed
$z_I$ lies inside the corresponding contour.
The same stability region slices for the new sixth-order IMEX-LMM in
\cref{tab:lmm6} are displayed in \cref{subfig:lmm6-implicit,subfig:lmm6-explicit,subfig:lmm6-imex}. More
precisely, \cref{subfig:lmm6-implicit} shows the purely implicit slice
$z_E=0$, \cref{subfig:lmm6-explicit} shows the purely explicit slice
$z_I=0$, and \cref{subfig:lmm6-imex} shows the $z_E$-plane slices for
fixed $z_I=0,-1,-5,-10$. In
\cref{subfig:lmm6-implicit,subfig:lmm6-explicit}, the stability regions are shaded in light blue,
while in \cref{subfig:lmm6-imex} the stability region for each fixed
$z_I$ lies inside the corresponding contour.

From \cref{fig:lmm6-stability}, we observe that the purely explicit stability
slices and the fixed-$z_I$ IMEX slices of IMEX-BDF6 in
\cref{subfig:bdf6-explicit,subfig:bdf6-imex} are comparable to those of the new IMEX-LMM6 in
\cref{subfig:lmm6-explicit,subfig:lmm6-imex}. The main difference appears in the purely
implicit slices shown in \cref{subfig:bdf6-implicit,subfig:lmm6-implicit}. Within the plotted
window, IMEX-BDF6 has a larger implicit stability region in terms of area,
whereas the new IMEX-LMM6 contains a wider sector around the negative real axis.
Thus, the new IMEX-LMM6 has a larger $A(\vartheta)$-stability angle,
where $A(\vartheta)$-stability means that the stability region contains
the sector $\{z=-r \mathrm{e}^{\mathrm{i}\phi} \in \mathbb{C}:\ r\geq 0,\ |\phi|\leq \vartheta\}$. 
A direct computation of the largest admissible sector angle gives
$\vartheta_{\rm BDF6}\approx 17.84^\circ$ and
$\vartheta_{\rm LMM6}\approx 26.15^\circ$.

\begin{figure}[htbp]
    \centering
    \begin{subfigure}{0.32\textwidth}
        \centering
        \includegraphics[width=\textwidth]{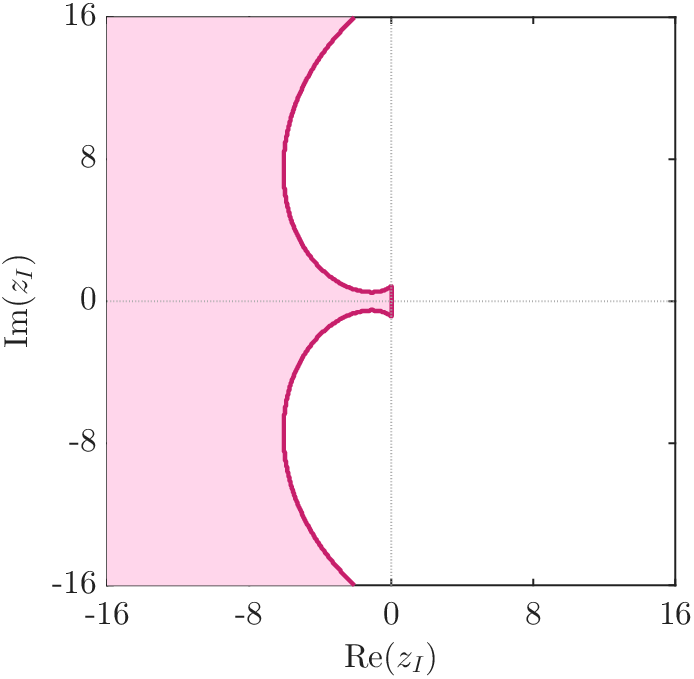}
        \caption{$z_I$-plane, $z_E=0$}
        \label{subfig:bdf6-implicit}
    \end{subfigure}
    \begin{subfigure}{0.32\textwidth}
        \centering
        \includegraphics[width=\textwidth]{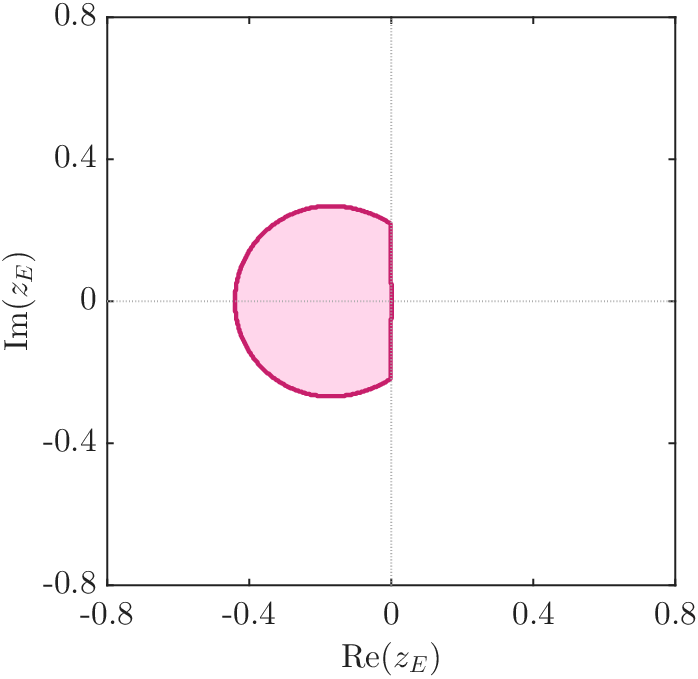}
        \caption{$z_E$-plane, $z_I=0$}
        \label{subfig:bdf6-explicit}
    \end{subfigure}
    \begin{subfigure}{0.32\textwidth}
        \centering
        \includegraphics[width=\textwidth]{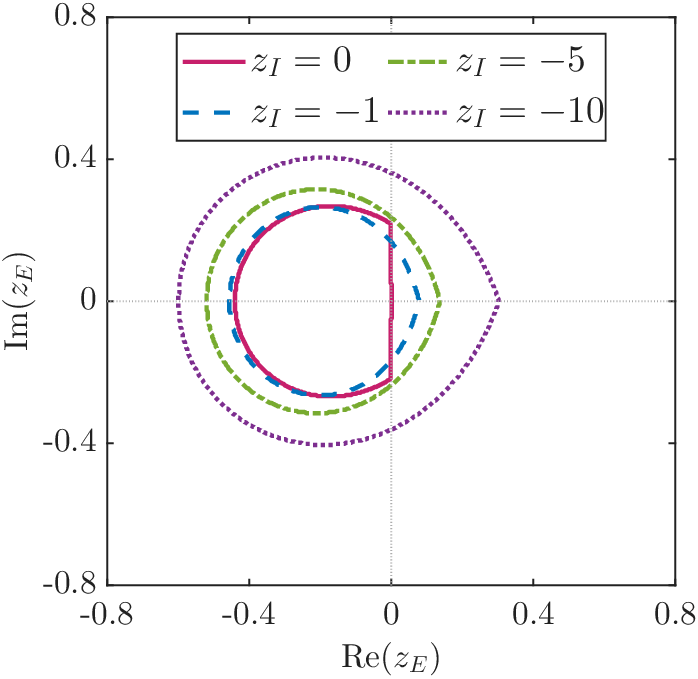}
        \caption{$z_E$-plane, fixed $z_I$}
        \label{subfig:bdf6-imex}
    \end{subfigure}\\
    \begin{subfigure}{0.32\textwidth}
        \centering
        \includegraphics[width=\textwidth]{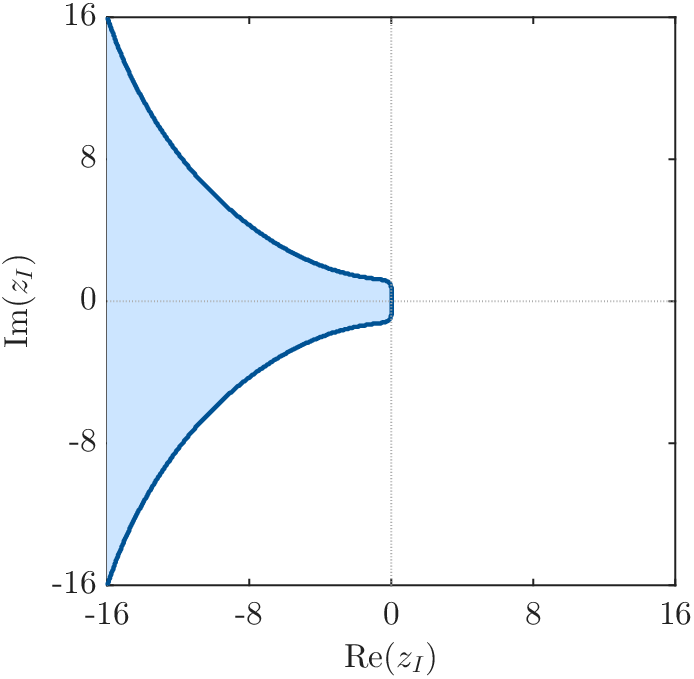}
        \caption{$z_I$-plane, $z_E=0$}
        \label{subfig:lmm6-implicit}
    \end{subfigure}
    \begin{subfigure}{0.32\textwidth}
        \centering
        \includegraphics[width=\textwidth]{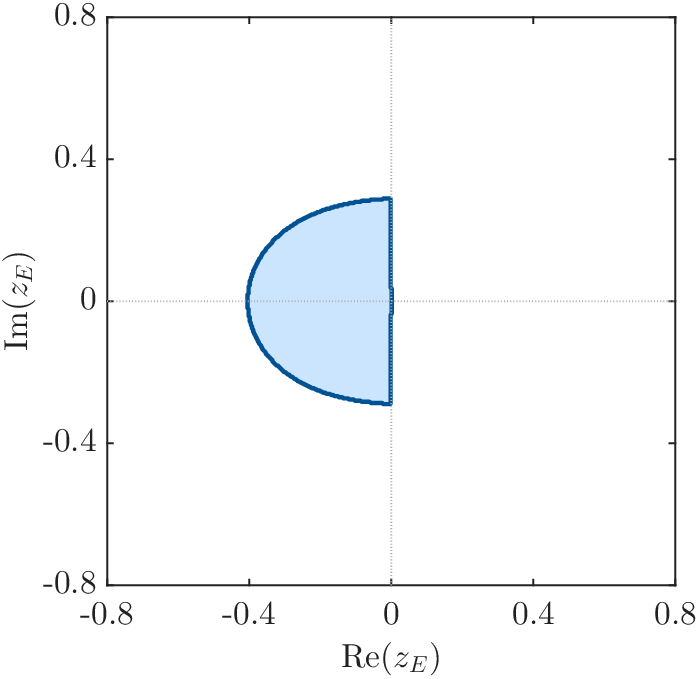}
        \caption{$z_E$-plane, $z_I=0$}
        \label{subfig:lmm6-explicit}
    \end{subfigure}
    \begin{subfigure}{0.32\textwidth}
        \centering
        \includegraphics[width=\textwidth]{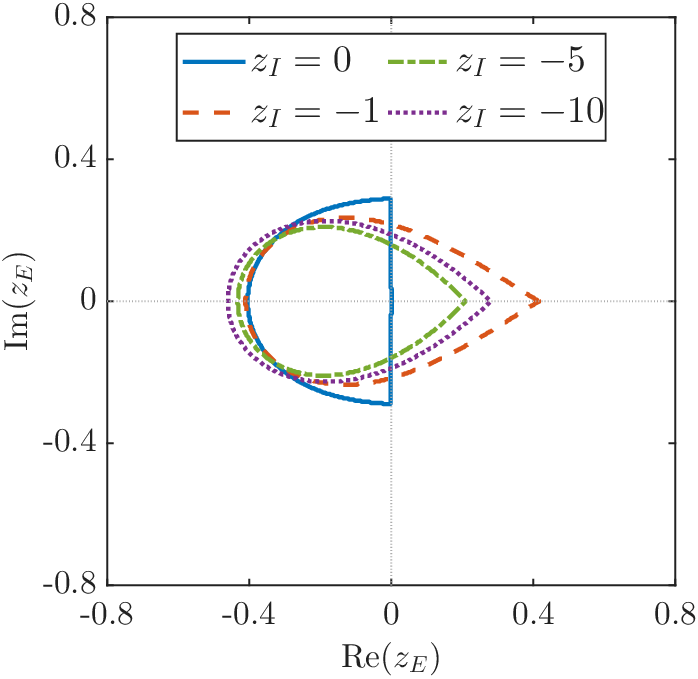}
        \caption{$z_E$-plane, fixed $z_I$}
        \label{subfig:lmm6-imex}
    \end{subfigure}
    \caption{Linear stability regions of IMEX-BDF6 (top row) and the proposed sixth-order IMEX-LMM listed in \cref{tab:lmm6} (bottom row): (a),(d) implicit stability regions with $z_E=0$; (b),(e) explicit stability regions with $z_I=0$; (c),(f) IMEX stability slices in the $z_E$-plane for fixed $z_I=0,-1,-5,-10$.}
\label{fig:lmm6-stability}
\end{figure}

\subsection{A sixth-order barrier}

While a family of sixth-order IMEX-LMMs exists, we will prove rigorously that no seventh-order IMEX-LMM satisfies the conditions in \eqref{eq:lr:es2}. That is, the feasibility problem \eqref{FP1} has no solution for $k=7$.

\begin{theorem}
There exists no seventh-order IMEX-LMM that satisfies the conditions in \eqref{eq:lr:es2}.
\end{theorem}
\begin{proof}
The nonexistence is equivalent to the infeasibility of the problem \eqref{FP1}. Since $T(x;\boldsymbol{a}^{(k)})$ and $T(x;\boldsymbol{b}^{(k)})$ are polynomials of degree at most $k-1$, a necessary condition for satisfying \eqref{eq:hom:obj} is that both polynomials are nonnegative at a set of $k$ distinct points in $[-1,1]$. We choose the Chebyshev points:
\[
x_j = \cos \frac{j \pi}{k-1}, \quad j=0,\ldots,k-1.
\]
The non-negativity of the polynomials at these nodes yields the discrete inequalities
\begin{equation}
\label{eq:lmm7:1}
Z \boldsymbol{a}^{(k)} \geq \boldsymbol{0},\quad  Z \boldsymbol{b}^{(k)} \geq \boldsymbol{0},
\end{equation}
where the inequalities are understood elementwise, and
\[
Z = 
\begin{bmatrix}
T_0(x_0) & T_1(x_0) & \cdots & T_{k-1}(x_0) \\
T_0(x_1) & T_1(x_1) & \cdots & T_{k-1}(x_1) \\
\vdots & \vdots & & \vdots \\
T_0(x_{k-1}) & T_1(x_{k-1}) & \cdots & T_{k-1} (x_{k-1})
\end{bmatrix}.
\]
By the definition of the Chebyshev polynomials, the entries of $Z$ are explicitly given by $Z_{j+1,m+1} = T_m(x_j) = \cos ({m j \pi}/(k-1))$ for $0\leq j,m \leq k-1$. 
Substituting \eqref{eq:hom:ab} into \eqref{eq:lmm7:1} yields
\[
Z  [E_k, \myvec{0}] W_1^{-1}
 \begin{bmatrix}
0 \\
\tilde{\boldsymbol{w}}
\end{bmatrix} \geq \boldsymbol{0},
\quad
Z E_k W_2^{-1} (D_k \tilde{\boldsymbol{w}} - w_k \boldsymbol{z}) \geq Z
\begin{bmatrix}
{1}/{2} \\
1 \\
\vdots\\
1
\end{bmatrix}.
\]
These inequalities can be written as $Q \boldsymbol{w} \leq \boldsymbol{q}$, where $Q=[Q_1^\top,Q_2^\top]^\top$, $\boldsymbol{q} = [\boldsymbol{q}_1^\top, \boldsymbol{q}_2^\top]^\top$, 
\[
\begin{aligned}
& Q_1 = 
-Z [E_k, \myvec{0}_{k\times 1}] W_1^{-1}
\begin{bmatrix}
\boldsymbol{0}_{2\times (k-1)} & \boldsymbol{0}_{2\times 1} \\
I_{k-1} & \boldsymbol{0}_{(k-1)\times 1} 
\end{bmatrix},
 \quad \boldsymbol{q}_1 = 
 Z [E_k, \boldsymbol{0}_{k\times 1}]  W_1^{-1}
 \begin{bmatrix}
 0 \\
 \boldsymbol{e}_1 \\
 \end{bmatrix}, \\
 &  Q_2 = 
  - Z E_k W_2^{-1} 
 \begin{bmatrix}
D_k, -\boldsymbol{z}
 \end{bmatrix}
 \begin{bmatrix}
 \boldsymbol{0}_{1\times k} \\
 I_k
 \end{bmatrix},
 \quad \boldsymbol{q}_2 = 
    Z E_k W_2^{-1}  \boldsymbol{e}_1
  -
  Z
\begin{bmatrix}
{1}/{2} \\
1 \\
\vdots\\
1
\end{bmatrix},
  \end{aligned}
\]
$\boldsymbol{e}_1 = [1,0,\ldots,0]^\top \in \mathbb{R}^k$, and $I_k$ is the identity matrix of size $k$. 
Consequently, the problem  \eqref{FP1} is infeasible if the following problem is infeasible:
\begin{enumerate}[label=\textnormal{(FP2)}, ref=FP2]
    \item \label{FP2} 
    Find $w_1, \dots, w_{k-1}, w_k \in \mathbb{R}$ such that
    \begin{equation*}
    Q\boldsymbol{w} \leq \boldsymbol{q}.
    \end{equation*}
\end{enumerate}

By the theorem of the alternative (also known as Farkas' Lemma; see, e.g., \cite[Sec. 5.8]{boyd2004convex}), the infeasibility of \eqref{FP2} is equivalent to the existence of a vector $\boldsymbol{\lambda} \in \mathbb{R}^{2k}$ such that 
\begin{equation}
\label{eq:lmm7:2}
\boldsymbol{\lambda} \geq \boldsymbol{0},\quad 
Q^\top \boldsymbol{\lambda} = \boldsymbol{0},\quad 
\boldsymbol{q}^\top \boldsymbol{\lambda} < 0.
\end{equation}

We now show that such a vector $\boldsymbol{\lambda}$ indeed exists when $k=7$. To this end, we first give the explicit forms of $Q_1, Q_2, \boldsymbol{q}_1,$ and $\boldsymbol{q}_2$:
\renewcommand{\arraystretch}{1.9}
{
\everymath{\displaystyle}
\[
Q_1 = 
\begin{bmatrix}
0 & 0 & 0 & 0 & 0 & 0 & 0 \\
 \frac{c_1}{720} & \frac{7c_2}{96} & \frac{c_3}{288} & \frac{7c_4}{480} & \frac{c_4}{1440} & 0 & 0 \\
\frac{91}{720} & \frac{1}{1440} & -\frac{35}{288} & -\frac{59}{1440} & -\frac{7}{1440} & -\frac{1}{5040} & 0 \\
\frac{649}{180} & \frac{35}{12} & \frac{8}{9} & \frac{7}{60} & \frac{1}{180} & 0 & 0 \\
\frac{1197}{80} & \frac{2237}{160} & \frac{189}{32} & \frac{201}{160} & \frac{21}{160} & \frac{3}{560} & 0 \\
\frac{\bar{c}_1}{720} & \frac{7\bar{c}_2}{96} & \frac{\bar{c}_3}{288} & \frac{7\bar{c}_4}{480} & \frac{\bar{c}_4}{1440} & 0 & 0 \\
-\frac{2156}{45} & -\frac{1708}{45} & -\frac{133}{9} & -\frac{136}{45} & -\frac{14}{45} & -\frac{4}{315} & 0
\end{bmatrix},
\quad 
\boldsymbol{q}_1 = 
\begin{bmatrix} 
1 \\
-\frac{7 c_5}{120}\\
\frac{403}{420} \\
-\frac{7}{15} \\
-\frac{333}{70} \\
-\frac{7\bar{c}_5}{120} \\
\frac{2416}{105}
\end{bmatrix},
\]
}
and
{
\everymath{\displaystyle}
\[
Q_2 =
\begin{bmatrix}
-\frac{1}{2} & 0 & 0 & 0 & 0 & 0 & 0 \\
\frac{7 c_5 }{240} & \frac{c_1}{2160} & \frac{7 c_2}{384} & \frac{c_3}{1440} & \frac{7c_4}{2880} & \frac{c_4}{10080} & 0 \\
-\frac{49}{120} & \frac{343}{2160} & \frac{31}{384} & \frac{7}{1440} & -\frac{1}{960} & -\frac{1}{10080} & 1 \\
\frac{7}{30} & \frac{649}{540} & \frac{35}{48} & \frac{8}{45} & \frac{7}{360} & \frac{1}{1260} & 0 \\
\frac{9}{20} & \frac{147}{80} & \frac{169}{128} & \frac{63}{160} & \frac{17}{320} & \frac{3}{1120} & -27 \\
\frac{7 \bar{c}_5}{240} & \frac{\bar{c}_1}{2160} & \frac{7\bar{c}_2}{384} & \frac{\bar{c}_3}{1440} & \frac{7 \bar{c}_4}{2880} & \frac{\bar{c}_4}{10080} & 0 \\
-\frac{104}{15} & -\frac{1148}{135} & -\frac{13}{3} & -\frac{49}{45} & -\frac{2}{15} & -\frac{2}{315} & 64
\end{bmatrix},\quad 
\boldsymbol{q}_2 = 
\frac{1}{2}
\begin{bmatrix} 
1 \\
1\\
1\\
1\\
1\\
1\\
1
\end{bmatrix},
\]
}
where the constants are defined as 
\[
\begin{aligned}
 c_1 &= -3734+2183 \sqrt{3},\ c_2 = -41+24 \sqrt{3},\ c_3 = -233+134 \sqrt{3}, \
 c_4  = -7+4 \sqrt{3},\\
c_5 &= -67+29 \sqrt{3}, \  \bar{c}_1 = -3734-2183 \sqrt{3}, \
\bar{c}_2 = -41-24 \sqrt{3},\\
\bar{c}_3  &= -233-134 \sqrt{3},
\ \bar{c}_4 = -7-4 \sqrt{3},\
\bar{c}_5 = -67-29 \sqrt{3}.
\end{aligned}
\]
Subsequently, we compute a basis of the kernel of $Q^\top$ and construct $\boldsymbol{\lambda}$ explicitly. 
We choose three vectors from the basis:  $\boldsymbol{r}^{(\ell)}=[r_1^{(\ell)},\ldots,r_{14}^{(\ell)}]^\top \in \mathbb{R}^{14}$ ($\ell = 1,2,3$),  whose entries are all zero except for the following:
{
\everymath{\displaystyle}
\[
\begin{array}{llll}
r^{(1)}_3 = \frac{-7+6 \sqrt{3}}{20},
& r^{(1)}_5 = \frac{37+94 \sqrt{3}}{540}, 
& r^{(1)}_7 = \frac{-13+24 \sqrt{3}}{640},   
& r^{(1)}_{13} = 1, \\
r^{(2)}_3 = \frac{2}{5}, 
& r^{(2)}_5 = -\frac{62}{135}, 
& r^{(2)}_7 = -\frac{11}{80},
& r^{(2)}_{11} = 1, \\
r^{(3)}_3 = \frac{-7-6 \sqrt{3}}{20}, 
& r^{(3)}_5 = \frac{37-94 \sqrt{3}}{540},
& r^{(3)}_7 = \frac{-13-24 \sqrt{3}}{640}, 
& r^{(3)}_9 = 1.
\end{array}
\]
}
Then, $Q^\top \boldsymbol{r}^{(\ell)} = \boldsymbol{0}$ for $\ell=1,2,3$. Let  
\begin{equation}
\label{eq:lmm7:lambda}
\boldsymbol{\lambda} = \boldsymbol{r}^{(1)} + \frac{1}{8}(3-\sqrt{3}) \boldsymbol{r}^{(2)} + (2-\sqrt{3}) \boldsymbol{r}^{(3)}.
\end{equation}
The resulting vector $\boldsymbol{\lambda} = [\lambda_1,\ldots,\lambda_{14}]^\top$ has only four non-zero entries:
\[
\lambda_5 = \frac{5}{27}(3-\sqrt{3}),\quad
\lambda_9 = 2-\sqrt{3}, \quad 
\lambda_{11} = \frac{1}{8}(3-\sqrt{3}), \quad 
\lambda_{13} = 1.
\]
A direct calculation confirms that the vector $\boldsymbol{\lambda}$ defined in \eqref{eq:lmm7:lambda} satisfies all conditions in \eqref{eq:lmm7:2}. 
Therefore, the problem \eqref{FP2} is infeasible, which implies the infeasibility of \eqref{FP1} for $k=7$. 
This completes the proof.
\end{proof}



\section{Numerical experiments}\label{sec:Numerical experiments}

This section presents several numerical experiments to demonstrate the convergence and energy dissipation properties of the proposed sixth-order IMEX-LMM, whose coefficients are listed in \cref{tab:lmm6}.
For all simulations, the spatial operators are discretized using the Fourier pseudo-spectral approximation on a $512\times512$ uniform mesh, 
with Fourier transforms computed by MATLAB's built-in \texttt{fft2} and \texttt{ifft2} routines \cite{matlabfft}. 
To obtain the first six starting values, a sixth-order Gauss collocation RK method is employed, and the resulting nonlinear system is solved via fixed-point iteration with a tolerance of $10^{-14}$ (see, e.g., \cite{GongZhao-GaussRK6,LiaoKang-2024IMA-PFC}).


\subsection{Convergence test}

\begin{example}\label{ex:accuracy-AC}
    Consider the Allen--Cahn equation with a source term $g(t, x, y)$, $u_{t} = \varepsilon^{2} \Delta u + u - u^{3} + g(t, x, y)$, defined on the domain $(0, 2\pi)^{2}$ with interface width $\varepsilon = 0.01$. The source term $g$ and the initial value $u(0, x, y)$ are determined by the exact solution $u(t, x, y) = \cos(t) \sin(x) \sin(y)$.
\end{example}
In the following convergence tests, the source term $g$ is also treated explicitly.
We first verify the temporal convergence rate of the proposed sixth-order IMEX-LMM method by setting the final time $T=1$. \cref{tab:convergence} summarizes the $L^\infty$ and $L^2$ norm errors, defined as $e_\infty(\tau) \coloneqq \max_{6\leq n \leq N}\|u^{n}-u(t_n)\|_\infty$ and $e_2(\tau) \coloneqq \max_{6\leq n \leq N}\|u^{n}-u(t_n)\|$, respectively. The errors are computed using various time steps $\tau=T/N$, where $N$ is taken from $\{25, 40, 50, 64, 80\}$ so that $\tau$ has a terminating decimal representation.
The Allen--Cahn test shows high-order convergence with observed rates slightly below six for this particular solution.

\begin{table}[htbp]
    \centering
    \setlength{\extrarowheight}{3pt}
    \begin{tabular}{c|cc|cc|cc|cc}
        \toprule
        \multirow{2}{*}{$\tau$} & \multicolumn{4}{c|}{Allen--Cahn equation in \cref{ex:accuracy-AC}} & \multicolumn{4}{c}{PFC equation in \cref{ex:accuracy-PFC}} \\
        \cline{2-9}
         & $e_\infty(\tau)$ & rate & $e_2(\tau)$ & rate & $e_\infty(\tau)$ & rate & $e_2(\tau)$ & rate \\
        \midrule
        $\frac{1}{25}$ & 3.654e-09 & -- & 1.348e-08 & -- & 1.433e-08 & -- & 4.521e-08 & -- \\
        $\frac{1}{40}$ & 2.863e-10 & 5.42 & 1.081e-09 & 5.37 & 9.123e-10 & 5.86 & 2.878e-09 & 5.86 \\
        $\frac{1}{50}$ & 8.208e-11 & 5.60 & 3.128e-10 & 5.56 & 2.378e-10 & 6.03 & 7.502e-10 & 6.03 \\
        $\frac{1}{64}$ & 2.025e-11 & 5.67 & 7.798e-11 & 5.63 & 5.321e-11 & 6.06 & 1.681e-10 & 6.06 \\
        $\frac{1}{80}$ & 5.753e-12 & 5.64 & 2.238e-11 & 5.59 & 1.370e-11 & 6.08 & 4.330e-11 & 6.08 \\
        \bottomrule
    \end{tabular}
    \caption{Errors and convergence rates of the sixth-order IMEX-LMM listed in \cref{tab:lmm6}.}
    \label{tab:convergence}
\end{table}

\begin{example}\label{ex:accuracy-PFC}
    Consider the PFC equation with a source term $g(t, x, y)$, $u_{t} = \Delta [(1+\Delta)^2 u + u^{3} - \varepsilon u] + g(t, x, y)$, subject to the initial condition $u(0, x, y) = \sin(x) \sin(y)$ on the domain $(0, 2\pi)^{2}$. We set the parameter $\varepsilon = 0.01$ and determine $g$ via the solution $u(t, x, y) = \cos(t) \sin(x) \sin(y)$.
\end{example}

The temporal convergence of the sixth-order IMEX-LMM is examined up to $T=1$. \Cref{tab:convergence} presents the $L^\infty$ and $L^2$ errors for a sequence of decreasing time steps $\tau=T/N$, where $N$ is taken from $\{25, 40, 50, 64, 80\}$ as before. The results confirm that our scheme achieves the expected sixth-order accuracy in time.

\subsection{Energy dissipation test}

\begin{example}\label{ex:energy-PFC}
    Consider the PFC equation without the source term in a large-scale domain $(0, 256)^{2}$ with $\varepsilon = 0.25$, following the setup in \cite{LiaoKang-2024IMA-PFC}. Periodic boundary conditions are employed to avoid artificial boundary effects. Three localized random perturbations are introduced on square patches of side length $10$, centered at $(64, 196)$, $(128, 64)$ and $(196, 196)$. The initial data is given by $u(0, x, y) = 0.285 + A(x,y) \cdot \text{rand}(x,y)$, 
    where $\text{\tt rand}(x,y)$ represents a uniform distribution in $(-1,1)$ generated with the fixed random seed $\text{\tt rng}(1)$, and $A(x,y)$ is the perturbation amplitude as in \cite{LiaoKang-2024IMA-PFC}:
    \[A(x,y) = \begin{cases}
        0.25, & \text{on the patch centered at } (64, 196), \\
        0.3, & \text{on the patch centered at } (128, 64), \\
        0.35, & \text{on the patch centered at } (196, 196), \\
        0, & \text{otherwise}.
    \end{cases}\]
\end{example}

The above initial condition $u^0$ has a nonzero average mass. This is consistent
with the zero-mean spaces used in the analysis, since the PFC equation
conserves the spatial average. Let $\bar u:=|\Omega|^{-1}\int_\Omega u^0\,\mathrm{d}\boldsymbol{x}$, where $|\Omega|$ denotes the area of $\Omega$. 
Then the shifted variable
$u-\bar u$ has zero mean, and the zero-mean spaces are used for this shifted
variable.
%
In \cref{ex:energy-PFC}, we use the splitting $\mathcal{L}=(\mathcal{I}+\Delta)^2+\mathcal{I}$ and $f(u)=u^3-(\varepsilon+1) u$. The corresponding potential function is $F(u)=\frac14(u^2-(\varepsilon+1))^2$. With this choice, the energy in our
formulation differs from the standard PFC energy 
\[
E_{\mathrm{PFC}}[u]
:=
\int_\Omega
\left[
\frac12 u(\mathcal{I}+\Delta)^2u
+\frac14u^4-\frac{\varepsilon}{2}u^2
\right]\mathrm{d}\boldsymbol{x}
\]
only by an additive constant
$C_0 \coloneqq \frac{(1+\varepsilon)^2}{4}|\Omega|$, which does not affect the modified energy dissipation law.

\begin{figure}[htbp]
\centering
    \begin{subfigure}{0.48\textwidth}
        \centering
        \includegraphics[width=\textwidth]{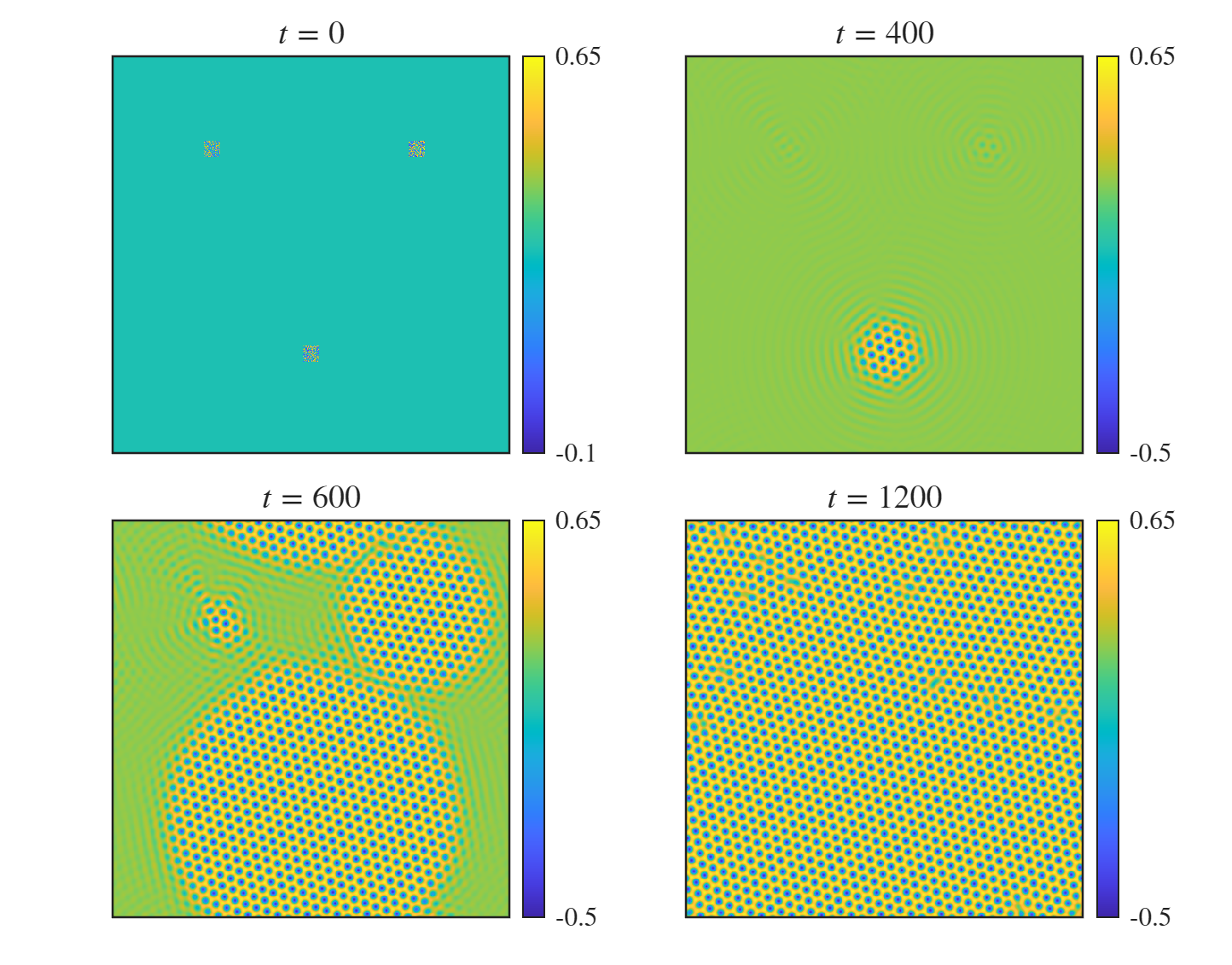}
        \caption{Snapshots of the crystal grain growth}
        \label{subfig:LMM6-PFC-snapshots}
    \end{subfigure}
    \quad
    \begin{subfigure}{0.4\textwidth}
        \centering
        \includegraphics[width=\textwidth]{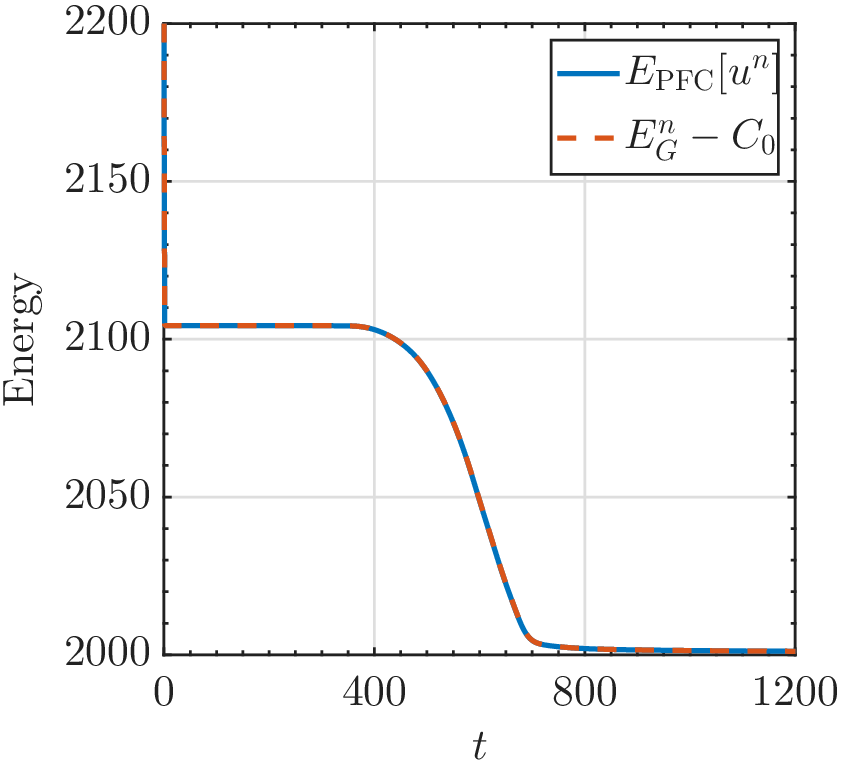}
        \caption{Energy dissipation}
        \label{subfig:LMM6-PFC-energy-dissipation}
    \end{subfigure}
    \caption{Snapshots of the crystal grain growth and energy dissipation of \cref{ex:energy-PFC}.}
\label{fig:LMM6-PFC-ex}
\end{figure}

The simulation is performed up to $T=3000$ with the time step $0.01$. Since the system reaches a steady state and the energy remains nearly constant after $T=1200$, the energy curves are displayed only over the interval $[0,1200]$ for better visualization. The snapshots of the phase field evolution are shown in \cref{subfig:LMM6-PFC-snapshots}. It is observed that the interface migration speed is positively correlated with the initial perturbation amplitude: the patch with the largest amplitude $(A=0.35)$ exhibits the fastest crystal growth. As the three grains expand, they eventually encounter each other and form stable grain boundaries. \Cref{subfig:LMM6-PFC-energy-dissipation} illustrates the evolution of the standard PFC energy and the modified energy. For a better representation, we plot $E_{\mathrm{PFC}}[u^n]=E[u^n]-C_0$ and $E_G^n-C_0$ with $C_0=25600$. This shift of $C_0$ does not affect the energy dissipation behavior.

Since the cubic nonlinearity is not globally Lipschitz, we choose a priori $R=2$ and then obtain $\ell_f=\max_{|s|\le R}|3s^2-(1+\varepsilon)| = 10.75$ for $\varepsilon=0.25$. This is the Lipschitz constant of $f(s)=s^3-(1+\varepsilon) s$ in $[-R,R]$; see this truncation technique in \cref{rem:truncationTech}. It is observed that the computed solution remains in this interval at all time levels and grid points, with $\max_{0\leq n\leq N}\|u^n\|_{\infty} \leq 0.7075 < 2$. Representative snapshots are shown in \cref{subfig:LMM6-PFC-snapshots}. Hence, the truncated and original nonlinearities coincide along the computed trajectory, and the modified energy dissipation property holds.


\section{Conclusion}\label{sec:conclusion}

In this work, we have established a theoretical framework for analyzing the energy dissipation of IMEX-LMMs applied to gradient flows. We derived a necessary and sufficient condition for the existence of the non-negative quadratic modification in  \eqref{eq:structure2}, which transforms the energy dissipation issue to the positivity of two real-valued generating polynomials over the interval $[-1, 1]$. This criterion serves as a straightforward tool to establish the energy dissipation property of general IMEX-LMMs, including the well-known IMEX-BDF$k$ methods. Further, a new sixth-order energy-dissipative IMEX-LMM is constructed within this framework, which appears to be the first, to the best of our knowledge. In addition,
a sixth-order barrier is proved rigorously. Numerical experiments on the Allen--Cahn and PFC equations have confirmed our theoretical findings regarding both the accuracy and energy dissipation.

\bibliographystyle{siamplain}
\bibliography{References}

\end{document}